\documentclass[twoside,11pt, preprint]{article}
 
 \usepackage{jmlr2e}
\RequirePackage{amsmath,amsfonts}
\usepackage{mathtools}
\usepackage{bbm}
\usepackage{booktabs}
\usepackage{tikz}
\usepackage{algorithm}
\usepackage{algpseudocode}
\usepackage{graphicx}

\newtheorem{defn}{Definition}
\newtheorem{thm}{Theorem}
\newtheorem{lem}{Lemma}
\newtheorem{assumption}{Assumption}
\newcommand{\eps}{\epsilon}

\newcommand{\ntmin}{N_t^{\text{min}}}
\newcommand{\ntmax}{N_t^{\text{max}}}

\newcommand{\fbb}{{f_{\boldsymbol{\beta}}}}
\newcommand{\fnb}{{f_{-\boldsymbol{\beta}}}}
\newcommand{\tfnb}{{\tilde{f}_{-\boldsymbol{\beta}}}}
\newcommand{\fbbp}{f_{\boldsymbol{\beta}'}}
\newcommand{\bbeta}{\boldsymbol{\beta}}
\newcommand{\Nbass}{N_{\text{bass}}}
\newtheorem{prop}{Proposition}

\jmlrheading{1}{2022}{1-48}{4/22}{4/23}{meila00a}{Azeem Zaman and Botond Szab\'o}

\ShortHeadings{Distributed nonparametric estimation}{Zaman and Szab\'o}
\firstpageno{1}

\begin{document}
\title{Distributed Nonparametric Estimation under Communication Constraints}

\author{\name Azeem Zaman \email azaman@g.harvard.edu \\
       \addr Department of Statistics\\
       Harvard University\\
       Cambridge, MA 02138, USA
       \AND
       \name Botond Szab\'o \email botond.szabo@unibocconi.it \\
       \addr Department of Decision Sciences, Bocconi University and 
Bocconi Institute for Data Science and Analytics (BIDSA)
       \\
       Milan, Italy}

\editor{Francis Bach and David Blei}

\maketitle

\begin{abstract}%
    In the era of big data, it is necessary to split extremely large data sets across multiple computing nodes and construct estimators using the distributed data. When designing distributed estimators, it is desirable to minimize the amount of communication across the network because transmission between computers is slow in comparison to computations in a single computer. Our work provides a general framework for understanding the behavior of distributed estimation under communication constraints for nonparametric problems. We provide results for a broad class of models, moving beyond the Gaussian framework that dominates the literature. As concrete examples we derive minimax lower and matching upper bounds in the distributed regression, density estimation, classification, Poisson regression and volatility estimation models under communication constraints. To assist with this, we provide sufficient conditions that can be easily verified in all of our examples. 
\end{abstract}






\section{Distributed Estimation}


As a consequence of the ever increasing amount of information being collected there has been substantial growth in distributed methods to handle data that is too large to store and analyze on a single machine. Distributed methods also arise naturally due to privacy considerations or scenarios when data are collected and processed at multiple locations before being summarized centrally. 

To statistically model a distributed estimation problem, we assume that we have $m$ machines or servers observing independent, identically distributed (iid) data. We will perform calculations on the local machines and transmit the results to a central machine or server for aggregation. In theory any type of information transmission is allowed, but in practice we typically transmit a local estimate (based on the local data). The local estimates are converted to a binary string and transmitted to a central machine that may aggregate these estimates in any manner. In this paper we focus on procedures that use only a single round of communication: the local machines transmit to the central machine only once and a final estimate is constructed based only on this single set of transmissions.

Communicating data across a network is orders of magnitude slower than working with the same data on a machine, forming the computational bottleneck of the problem. This means that methods communicating large amounts of data may be too slow to be practically useful. In addition, excess communication can cause congestion on the computing network and inconvenience other users.  Therefore we introduce communication constraints in our network  to minimize the run time of the algorithm. In the absence of such communication constraints, one may transfer the whole data set to the central machine and the distributed framework imposes no meaningful restriction. We assume that each machine from $k = 1,\ldots, m$ has a communication budget $B^{(k)}$, which is a limit on the length of the binary string that can be transmitted to the central machine. 
Let $\mathcal{F}_{\text{dist}}(\boldsymbol{B})$, where $\boldsymbol{B} = (B^{(1)}, \ldots, B^{(m)})$, denote set of estimators that can be constructed with a single round of communication subject to the budget constraint. Our goal is to determine the minimax rate for estimators in $\mathcal{F}_{\text{dist}}(\boldsymbol{B})$ in an abstract statistical framework and apply the derived results on various specific models.

The theoretical study of statistical estimation and testing under communication constraints has emerged only recently, with most of the work focusing on parametric models, see for instance \cite{zhang2013information,duchi_optimality_2014,shamir_fundamental_2014,braverman_communication_2016,xuraginsky2016,fan2019distributed,han_geometric_2020,
cai_distributed_2020,cai:distributed:adap:sigma} for estimation and \cite{9211418,pmlr-v125-acharya20b,szabo_optimal_2020,sz:vuursteen:zanten:2022} for testing. Nonparametric models, where the parameter of interest is infinite dimensional, have been studied in just the last few years in context of specific models, focusing mainly on Gaussian likelihood functions.  Distributed minimax estimation rates were derived in context of the Gaussian white noise \cite{pmlr-v80-zhu18a}, and nonparametric regression models \cite{szabo_adaptive_2019}. Adaptation to the unknown regularity parameter was investigated in \cite{szabo_adaptive_2019,szabo2020distributed,cai:2021:distributed}. Outside of the Gaussian framework only the closely related density estimation problem \cite{barnes_lower_2019} has been investigated so far. Thus up to now there is a lack of a general, abstract understanding of the limitations and guarantees of distributed techniques under communication constraints in nonparametric settings. 

In this paper we take on this task and derive minimax lower and matching upper bounds in abstract settings going beyond the Gaussian framework, requiring, roughly speaking, only appropriate tail bounds on the likelihood ratio over an appropriately selected finite sieve. We demonstrate the applicability of our results over a range of nonparametric models, including nonparametric regression, density estimation, binary and Poisson regressions and volatility estimation.

We organize the paper as follows: In Section \ref{sec:main} we formally introduce distributed methods and our abstract modeling assumptions under which minimax lower and matching upper bounds are derived in Sections \ref{sec:lower} and \ref{sec:upper}, respectively. In Section \ref{sec:suff-conds} we provide simple sufficient conditions that we can be used to apply our general theorems to the examples presented in Section \ref{sec:examples}. We summarize our results and discuss further research questions in Section \ref{sec:discuss}. The proofs for the minimax lower and upper bounds are presented in Section \ref{sec:proofs}, while in Section \ref{sec:proof-sufficient} we demonstrate the validity of our sufficient conditions. The proofs of the examples are given in Section \ref{sec:proof:examples}, while auxiliary lemmas and their proofs are deferred to the Supplement.

\subsection{Notation}
For two positive sequences $a_n$, $b_n$ we use the notation $a_n\lesssim b_n$ if
there exists an universal positive constant $C$ such that $a_n \leq  Cb_n$. Then we denote by $a_n\asymp b_n$ if both $a_n\gtrsim b_n$ and $b_n\lesssim a_n$ hold simultaneously. Furthermore, we write $a_n\ll b_n$ if $a_n/b_n=o(1)$. We denote by $\lceil b\rceil$ and $\lfloor b \rfloor$ the upper and lower integer part of $b\in\mathbb{R}$, respectively. For a set $I$ we denote by $|I|$ the cardinality of $I$. Throughout the paper, $c$ and $C$ denote global constants whose value may change from one line to another.

\section{Main results}\label{sec:main}

The goal of this paper is to derive theoretical limitations and guarantees for distributed methods under communication constraints in a general, abstract setting. We assume that there are $m$ machines and in each machine $k=1,...,m$ we observe a sample  $X^{(k)}=\{ X_{i}^{(k)}\}_{1 \leq i \leq n}$ from some unknown distribution $p_f$ indexed by an infinite dimensional functional parameter $f\in L_2([0,1])$ of interest. The total (global) sample size is denoted by $N=nm$. We assume, that the function $f$ belongs to a Besov (Sobolev) ball $f\in B_{2,2}^r(L)$ for some known $r,L>0$, see Section \ref{app: wavelets} in the Supplement for a brief summary of Besov spaces and wavelets.  As an example consider the distributed nonparametric regression model with observations $(T_i^{(k)},X_i^{(k)})$, $i=1,...,n$, $k=1,...,m$,  satisfying that $T_i^{(k)}\sim^{iid} U(0,1)$ and $X_i^{(k)}|T_i^{(k)}\sim^{iid}  \phi(x - f(T_i^{(k)}))$, studied also in \cite{szabo_adaptive_2019}.

We consider estimation procedures where each machine can transmit a binary message $Y^{(k)}$ with almost sure length constraint $l(Y^{(k)}) \leq B^{(k)}$ to the central machine where the information is aggregated to form a global estimator $\hat{f}_{n,m}=\hat{f}_{n,m}(Y^{(1)},...,Y^{(m)})$ depending only on the transmitted information; see display \eqref{fig: dist:est} below for a graphical representation.

\begin{equation} \label{fig: dist:est}
 \begin{matrix}
X^{(1)} &\put(0,5){\vector(1,0){20}}  &\qquad Y^{(1)}\quad&\put(-10,4){\vector(3,-1){20}} \\
\vdots &\put(0,5){\vector(1,0){20}}  &\qquad\vdots  \quad&\put(-10,5){\vector(1,0){20}} \\
X^{(m)} & \put(0,5){\vector(1,0){20}}  &\qquad Y^{(m)}\quad&\put(-10,5){\vector(3,1){20}}
\end{matrix}  \qquad
 \hat{f}_{n,m}.
\end{equation}

We derive minimax lower bounds under abstract conditions on the likelihood ratio in an appropriately selected sieve and show the sharpness of the lower bounds by providing matching upper bounds (up to $\log n$ factor) under non-restrictive conditions. Before proceeding to the lower bounds we quantify the information content in the communication-restricted, distributed setting for recovering the functional parameter of interest $f$.  For this end we define the \emph{bit adjusted sample size} (BASS) $N_{\text{bass}}=N_{\text{bass}}(B^{(1)}, \ldots, B^{(m)})$ for estimating $f\in B_{2,2}^r$ as the solution to the following equation:
\begin{align}\label{eq:def-bass-lower}
    N_{\text{bass}} &= \max\Big(n\log{N}, n\sum_{k=1}^m \big(\frac{B^{(k)}\log{N}}{\Nbass^{1/(1 + 2r)}} \wedge 1 \big)\Big).
\end{align}
The exact form  of $\Nbass$ will be verified below. The equation in \eqref{eq:def-bass-lower} does not have a closed form solution in general, but the BASS is always defined since the left-hand side of \eqref{eq:def-bass-lower} is increasing in $\Nbass$ and the right-hand side is decreasing in $\Nbass$. The expression can be greatly and easily simplified if we assume that $B^{(1)} = B^{(2)} = \dots = B^{(m)}=B$, which is presented below.

\begin{prop}[$\Nbass$ in symmetric case] \label{cor:Nbass-symmetric}
If $B^{(k)} = B$ for all $k = 1, \ldots, m$, then
\begin{align}\label{eq:Nbass-symmetric-defn}
\Nbass &= \begin{cases} N, & \text{if}\quad B \geq N^{1/(1+2r)}/\log{N}, \\
\left(NB\log{N}\right)^{\frac{1+2r}{2+2r}}, &\text{if}\quad \left(n\log{N}\right)^{1/(1+2r)}/m \leq B < N^{1/(1+2r)}/\log{N}, \\
n\log{N}, &\text{if}\quad B < \left(n\log{N}\right)^{1/(1+2r)}/m. \end{cases}
\end{align}
\end{prop}

We briefly discuss the above proposition. Based on \eqref{eq:Nbass-symmetric-defn}, we see that in the symmetric budget case the BASS is a continuous function of the budget $B$. Furthermore, note that it depends on the smoothness $r$ of the function $f$. For more regular functions communication constraints results less information loss. Also note that if the communication budget is large enough, then $\Nbass=N$ and there is no information loss due to the distributed architecture. However, if the communication budget is very small (i.e. $ B < \left(n\log{N}\right)^{1/(1+2r)}/m$) then there are no benefits in communicating between machines and it is sufficient or even advantageous to use only the information available on a single machine, assuming that the central machine also receives a local data set of size $n$. In the intermediate case, $\Nbass $ depends on the communication budget $B$, the smoothness level $r$, and the overall sample size $N$.

\subsection{Lower bounds}\label{sec:lower}

In this section we present lower bounds for the distributed minimax estimation rate in a general setting. To prove lower bounds one typically considers an appropriately chosen finite sieve $\mathcal{F}_0\subset \mathcal{F}$ of the function class $\mathcal{F}= B_{2,2}^r(L)$ of interest, i.e. note that
\begin{align*}
\adjustlimits{\inf}_{\hat{f}\in \mathcal{F}_{\text{dist}}(\boldsymbol{B})}{\sup}_{f\in \mathcal{F}}\mathbb{E}_{f}d(f,\hat{f})
\geq \adjustlimits{\inf}_{\hat{f}\in \mathcal{F}_{\text{dist}}(\boldsymbol{B})}{\sup}_{f\in {\mathcal{F}_0}}\mathbb{E}_{f}d(f,\hat{f}),    
\end{align*}
for arbitrary loss function $d: \mathcal{F}\times \mathcal{F} \mapsto [0,\infty)$. For Besov regularity classes it is natural to define the finite sieve with the help of the wavelet basis functions. For resolution level $j\in\mathbb{N}$, we define
\begin{align*}
\tilde{\mathcal{F}}_0(j,C_0,\eps)=\Big\{\fbb(t)+C_0 :\, \fbb(t)=   \epsilon 2^{-j(r + 1/2)}\sum_{h =0 }^{2^j-1}\beta_h\psi_{jh}(t)\,:\, \beta\in\{-1,1\}^{2^j}\Big\},
\end{align*}
for parameters $C_0,\eps\in\mathbb{R}$. The parameter $C_0$ shifts the function to ensure that the modeling assumptions are met; certain models, such as density estimation (where $C_0 + \fbb$ is the density) or Poisson regression (where $C_0 + \fbb$ is the mean of a Poisson), require $C_0 + \fbb(t) \geq 0$.  In our analysis we consider compactly supported wavelets (e.g. boundary corrected Daubechies wavelets with $S>r$ vanishing moments), see Section \ref{app: wavelets} for a brief review. For convenience we consider a subset of the basis functions at resolution level $j$,  such that they have pair-wise disjoint supports.  Therefore, we divide the interval $[0,1]$ into a partition of $d = c2^j$ (for some constant $0<c<1$ depending on $r$) disjoint intervals $I_1, \ldots, I_{d}$, such that each interval $I_k$ contains the full support of a wavelet basis function $\psi_{j,h}$ for  $h \in \{0,\ldots 2^j - 1\}$. For Daubechies wavelets with $S$ vanishing moments, this is possible for $c \leq 1/(2S + 2)$. By slightly abusing our notations we re-index the basis functions so that the subset of size $d$ of wavelets with disjoint support are indexed by $h = 1,\ldots, d$.  Hence we consider finite sieves of the form
\begin{align}\label{def: sieve_disjoint}
\mathcal{F}_0(d,C_0,\eps)=\Big\{\fbb(t)+ C_0: \,   \fbb(t) = \epsilon d^{-(r + 1/2)}\sum_{h =1 }^{d}\beta_h\psi_{jh}(t)\,:\, \beta\in\{-1,1\}^{d}\Big\},
\end{align}
where $\psi_{jh}$ and $\psi_{jh'}$ have disjoint support for any $1\leq h<h'\leq d$. For a fixed $d$, we will take wavelets at resolution $j$, where $j$ is the smallest resolution level that allows for the construction of $d$ disjoint wavelets. 

We give our assumption in terms of the log-likelihood-ratios in the sieve $\mathcal{F}_0(d,C_0,\eps)$. Let us denote the log-likelihood ratio between two elements $\fbb(t)+ C_0,f_{\boldsymbol{\beta}'}(t)+ C_0\in \mathcal{F}_0(d,C_0,\eps)$ for a single observation $X_i^{(k)}$  by 
\begin{align*}
    W_{d}^{(k)}(X_i^{(k)}, \bbeta, \bbeta') &= \log\frac{p_{\fbb}(X_i^{(k)})}{p_{\fbbp}(X_i^{(k)})}. 
\end{align*}
We assume below that the log-likelihood-ratio for the observations is sufficiently small, both in expectation and with high-probability, and that it satisfies a nonrestrictive decomposition structure. This is formalized in the following assumptions.

\begin{assumption}[Expectation control] \label{ass:W-expect-control} Assume that there exists a constant $C>0$ such that
$$\mathbb{E}_{ \bbeta}[W_{d}^{(k)}(X_i^{(k)},  \bbeta,  \bbeta') ] \leq C \eps^2 d^{-2r}.$$
\end{assumption}

\begin{assumption}[Decomposable structure] \label{ass:structural} Assume that the sample space can be partitioned into events $\chi_1,\ldots, \chi_d$ such that  $\mathbb{P}_{\bbeta}(\chi_h)=1/d$ for all $\bbeta\in\{-1,1\}^d$, and for all $i=1,...,n$, $k=1,...,m$ and $h=1,...,d$ it holds that the distribution $X_{i}^{(k)}|(X_{i}^{(k)}\in \chi_h)$ only depends on $\beta_h$ and not on $\beta_{h'}$ for $h'\neq h$.
\end{assumption}

\begin{assumption}[Tail control] \label{ass:tail} Assume that for $a_n = C\eps d^{-1/2-r}\sqrt{n\log(N)}$ the event
$$    E_h^{(k)}(a) = \left\{ X^{(k)}\,:\, \max_{\bbeta,\bbeta'} \sum_{X_{i}^{(k)}\in\chi_h} W_{d}^{(k)}\left(X_{i}^{(k)}, \bbeta, \bbeta'\right)\leq  a \right\}$$
satisfies 
$$\mathbb{P}_{\beta_h}\left(E_h^{(k)}(a_n)\right)=1+o(m^{-2}d^{-2})\qquad \forall\, h\in\{1,...,d\},\,\, k\in\{1,...,m\},\,\, \beta_h\in\{-1,1\}.$$
\end{assumption}

\begin{remark}
In Assumption \ref{ass:structural} one can consider arbitrary fixed distribution over the partition  events $\chi_{h}$, $h=1,...,d$, which for simplicity we take to be uniform. Our results can be easily generalised to other distributions over the partition. 
\end{remark}

We provide a brief discussion of the above assumptions. Assumption \ref{ass:W-expect-control} is equivalent to bounding the KL divergence between $p_\fbb$ and $p_{\fbbp}$, which is a standard tool used for deriving minimax lower bounds in various nonparametric models, see for instance Theorem 6.3.2 of \cite{Gine}. Assumption \ref{ass:structural} means that the dependence on the functional parameter must be spatially localized, which is true for many common models. When working with a wavelet based testing sieve, we can select the partition based on the supports of the wavelets used in $\fbb$. Finally, in Assumption \ref{ass:tail} we require that the maximum of the likelihood ratio over the sieve $\mathcal{F}_0(d,C_0,\eps)$ evaluated at the data belonging to bin $I_h$ is bounded (by $a_n$) with high probability. While the first assumption controls the mass (expected value) of the likelihood ratio, this assumption controls the tail behaviour. Based on these assumptions we can formulate our general lower bound theorem.

\begin{thm} \label{thm:gen_lower_bound_thm}
Suppose Assumptions \ref{ass:W-expect-control}, \ref{ass:structural}, and \ref{ass:tail} hold. The distributed minimax rate with communication constraints $\boldsymbol{B} = (B^{(1)},\ldots, B^{(m)})$ for recovering $f \in B_{2,2}^r(L)$ is bounded from below by
\begin{align*}
    \inf_{\hat{f} \in \mathcal{F}_{\text{dist}}(\boldsymbol{B})} \sup_{f \in B_{2,2}^r(L) } \mathbb{E}_{f}\Vert \hat{f} - f\Vert_2^2 \gtrsim \Nbass^{-\frac{2r}{1+2r}}.
\end{align*}
\end{thm}
The proof is deferred to Section \ref{sec:proof-lower-bound}, but we provide a few comments here. The proof is an extension of the common technique for deriving minimax bounds using Fano's inequality. The general technique follows the ideas of \cite{zhang2013information}, which studied parametric models. These techniques were extended to the nonparametric regression model in \cite{szabo_adaptive_2019}. In non-distributed models the main step is to control the expected value of the log-likelihood-ratio of two distributions in the testing sieve (the quantity $\cramped{W_{d}^{(k)}}$ discussed above). The difficulty of extending to distributed architectures with communication constraints is that we must control not just the expectation of the likelihood ratio (Assumption \ref{ass:W-expect-control}), but also the tail of the distribution of the ratio (Assumption \ref{ass:tail}).



\subsection{Upper bounds}\label{sec:upper}

Here we introduce conditions that ensure that the general lower bounds derived in Theorem \ref{thm:gen_lower_bound_thm} are tight by providing matching upper bounds (up to log-factors). We focus on the case where the communication budget across machines are all equal, $B = B^{(1)} = \dots = B^{(m)}$. We also assume that the central machine receives a local sample of size $n$. Then if the communication budget is too small (i.e.  $B<\left(n\log{N}\right)^{1/(1+2r)}/m$), then a standard local estimator (based solely on the local data available at the central machine) achieves the minimax lower bounds $n^{-r/(1+2r)}$. In this case distributed methods are not adventageous compared to methods considering only local information hence our focus lies on the more interesting large enough budget case (i.e. $B\geq \left(n\log{N}\right)^{1/(1+2r)}/m$).

We will estimate the wavelet coefficients $f_{jh}$ of $f(x)=\sum_{j=0}^{\infty}\sum_{h = 0}^{2^j - 1}f_{jh}\psi_{jh}(x)$ using local estimators $\hat{f}_{n,jh}^{(k)}$ on all machines $k = 1,\ldots, m$. The estimator $\hat{f}_{n,jh}^{(k)}$ should be thought of as the local version of the standard estimator of the wavelet coefficient $f_{jh}$ that would be used in a non-distributed setting. We then compress $\cramped{\hat{f}_{n,jh}^{(k)}}$ of $f_{jh}$ into a $\log_2{N}$ long binary string $\cramped{Y_{jh}^{(k)}}$ following Algorithm 1 from \cite{szabo_adaptive_2019}, which is reproduced as Algorithm \ref{alg:transapprox} in the Supplement for reference. When $|\hat{f}_{n,jh}^{(k)}| > \sqrt{N}$, we transmit $Y_{jh}^{(k)} = 0$ to reduce communication costs. We denote the binary approximation of the estimators $\hat{f}_{n,jh}^{(k)}$ by $\tilde{f}_{n,jh}^{(k)}=\tilde{f}_{n,jh}^{(k)}(Y_{jh}^{(k)})$. We provide below general assumptions on the local estimators to ensure that upper bounds match (up to a logarithmic factor) the lower bounds derived in the preceding section.



\begin{assumption}\label{ass:UB}
There exist $\sigma^2>0$ such that the local wavelet coefficient estimators $\hat{f}_{n,jh}^{(k)}$, $j\leq \log (N)/(1+2r)$, $h=0,...,2^j-1$ satisfy
\begin{enumerate}
    \item $\mathbb{E}_f[\hat{f}_{n,jh}^{(k)}] = f_{jh}$, $\text{var}(\hat{f}_{n,jh}^{(k)}) \leq \sigma^2/n$,
    \item $\mathbb{P}_f(\hat{f}_{n,jh}^{(k)} > \sqrt{N}) \leq o( m^{-1}N^{-(3+4r)/(1+2r)})$.
\end{enumerate}

\end{assumption}
We briefly discuss the assumptions. The condition on the mean and variance of the non-zero  estimators are standard. The tail condition is necessitated by the communication constraint. In order to ensure that our estimation procedure satisfies the communication constraint, we do not transmit estimators larger than $\sqrt{N}$ (see Algorithm \ref{alg:block-est-alg} in the Supplement) as this would require too many bits. Since we truncate large estimators before transmitting them, we need to control the probability that an estimator is truncated in order to control the error introduced by this truncation. In all examples considered in this paper, the local estimator $\hat{f}_{n,jh}^{(k)}$ is an  average which converges almost surely to $f_{jh}$ by the law of large numbers. In light of this, the tail condition is mild. Section \ref{sec:sufficient:UB} gives sufficient conditions for the tail condition in terms of existence of higher order moments.


When the communication budget is too small, we must select a subset of wavelet coefficients for each machine to estimate. Our choice will extend the construction given in \cite{szabo_adaptive_2019} in context of the nonparametric regression model to our abstract setting. We divide the machines into groups of size $\kappa$, where
\begin{align}
    \kappa = \max\left(1, \left\lfloor\left(\frac{B}{\log_2{N}}m\right)^{\frac{1+2r}{2+2r}}n^{-\frac{1}{2+2r}} \right\rfloor\wedge m\right). \label{eq:optimal-kappa}
\end{align}
Machines in different blocks will transmit the binary strings $Y_{jh}^{(k)}$ corresponding to different  wavelet coefficients. Machines in the $\ell$th block, i.e.  machines $k\in \{(\ell-1)\kappa+1,...,\ell\kappa\}$ with $\ell\in \{1,...,\lceil m/\kappa\rceil\}$, will transmit the coefficients in the  $\ell$th block, i.e.  transmit $Y_{jh}$ with indices satisfying $2^j + h = (\ell - 1)\lfloor B/\log_2{N}\rfloor$ to $\ell\lfloor B/\log_2{N}\rfloor$. For all other values of $j$ and $h$ we set $\hat{f}_{n,jh}^{(k)} = 0$ and transmit nothing. See Figure \ref{fig:block-ests} for graphical representation.

The aggregated estimator for the $(j,h)$th wavelet coefficient with  $2^j + h = (\ell - 1)\lfloor B/\log_2{N}\rfloor$ to $\ell\lfloor B/\log_2{N}\rfloor$ is defined as
\begin{align*}
    \hat{f}_{jh} &= \frac{1}{\kappa}\sum_{k=(\ell-1)\kappa+1}^{\ell\kappa} \tilde{f}_{n,jh}^{(k)}
\end{align*}
and the corresponding estimator of the functional parameter is 
\begin{align}
\hat{f}_{n,m}(\cdot)=\sum_{j=0}^{j_n}\sum_{h=0}^{2^j-1} \hat{f}_{jh}  \psi_{jh}(\cdot),\label{eq: block:est}
\end{align} 
with $j_n = \lfloor \log_2{(n\kappa)}/(1 + 2r)\rfloor$ being the maximum resolution level we will estimate with the given budget. The theorem below states that this estimator achieves the minimax distributed rate up to a logarithmic factor.

\begin{figure}
\centering
     \begin{tikzpicture}
    \foreach \x in {0,1,...,18}
    \foreach \y in {0,1,2,3,4,5,6}
    {
    \draw (.5*\x,.5*\y) circle (0.1cm);
    }
    \foreach \x in {0,1,...,8}
    \foreach \y in {6,5,4}
    {
    \draw[fill = gray] (.5*\x,.5*\y) circle (0.1cm);
    \draw[fill = gray] (.5*\x +4.5, .5*\y-1.5) circle (0.1cm);
    }
    \draw[<->] (0, -0.5) -- (9,-0.5);
    \draw (4.5, -0.7) node {$N^{\frac{1}{1 +2r}}$};
    \draw[<->] (9.5,0) -- (9.5, 3);
    \draw (9.8, 1.5) node {$m$};
    \draw[<->] (0,3.4) -- (4,3.4);
    \draw (2, 3.7) node {$B/\log_2{N}$};
    \draw[<->] (-0.5,3) -- (-0.5,2);
    \draw (-0.7, 2.5) node {$\kappa$};
    \draw[<->] (0, -.95) -- (8.5,-.95);
    \draw (4.25,-1.29) node {$\Nbass^{\frac{1}{1+2r}}$};
  \end{tikzpicture} 
    \caption{Diagram of block estimators: columns correspond to wavelet coefficients and rows represent $m$ machines. Each machine can transmit at most $B/\log_2{N}$ coefficient estimates (with negligible approximation error). The gray circles represent the coefficients estimated and transmitted at a given machine.  \label{fig:block-ests}} 
\end{figure}
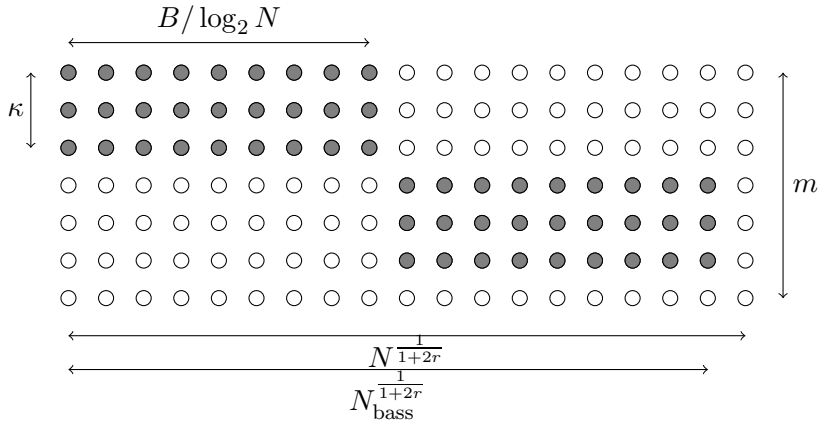

\begin{thm}[Upper bound on minimax rate]\label{thm:gen_upper_bound_thm}
Under Assumption \ref{ass:UB} the distributed estimator $\hat{f}_{n,m}\in\mathcal{F}_{distr}(\textbf{B})$  defined in \eqref{eq: block:est} satisfies
\begin{align*}
    \sup_{f \in B_{2,2}^r(L)} \mathbb{E}_{f} \Vert \hat{f}_{n,m} - f \Vert_2^2 \lesssim ( \Nbass/\log N )^{-\frac{2r}{1+2r}}.
\end{align*}
\end{thm}                                                               

\begin{remark}
The $\log N$ factor in the upper bound is not optimized and in some regimes our estimator performs better, for instance for large enough budget (i.e. $N^{1/(1+2r)}\log{N} \leq B$) it can be completely removed.
\end{remark}


Assumption \ref{ass:UB} can be easily verified for many models using the sufficient conditions given in Section \ref{sec:sufficient:UB}.

\subsection{Sufficient conditions for regression type problems} \label{sec:suff-conds}
In this section we consider ``regression type'' models and provide sufficient conditions for upper and lower bounds for them. We assume that there are $m$ machines and in each machine $k=1,...,m$ we observe a sample $(T^{(k)},X^{(k)})=\{( T_{i}^{(k)}, X_{i}^{(k)})\}_{1 \leq i \leq n}$ whose density $p_f(x,t)$ factors as
\begin{align}
    X_i^{(k)} \mid T_{i}^{(k)} \sim p_{X \mid T}(x \mid f(T_{i}^{(k)})),\quad T_i^{(k)} \sim p_T(t),\quad i=1,...,n,\quad k=1,...,m,\label{eq:model:reg}
\end{align}
for densities $p_T(t)$ and $p(x \mid f(t))$. For simplicity we take  $p_T(t)$ to be the uniform density on $[0,1]$, but it can be easily extended to bounded, positive densities in compact intervals. In this more structured model it is possible to present sufficient conditions that are easier to verify in many examples.

\subsubsection{Sufficient conditions for lower bounds}
In this section we collect sufficient conditions for Assumptions \ref{ass:W-expect-control}, \ref{ass:structural} and \ref{ass:tail} when we are considering models of the form \eqref{eq:model:reg}. These conditions consider a more restrictive class of models, but are typically easier to verify in a range of our examples. Without these results, verifying Assumption \ref{ass:tail} requires us to establish a complicated tail bound on the likelihood-ratio (rather than the random variable itself). First let us introduce some short hand notations for the log-likelihood-ratio, which decomposes into a sum under \eqref{eq:model:reg}:
\begin{align*}
    W_{d}^{(k)} =\sum_{i=1}^n W_{d,i}^{(k)}= \sum_{i=1}^n W_{d}^{(k)}(X_{i}^{(k)}, T_{i}^{(k)}, \boldsymbol{\beta}, \boldsymbol{\beta'}).
\end{align*}


First we show that in the model \eqref{eq:model:reg} Assumption \ref{ass:structural} automatically holds. Consider the partition $\chi_h=I_h \times \mathbb{R}$ for $h = 1,\ldots, d$ and note that the event $(T_i^{(k)},X_i^{(k)})\in \chi_h$ holds with probability $1/d$. Furthermore, conditional on the preceding event, the random design point $T_i^{(k)}$ is uniformly distributed over $I_h$, while $X_i^{(k)}|T_i^{(k)}\sim p_{X|T}\big(\cdot |f_{\bbeta}(T_i^{(k)})\big)$ only depends on the value of $\beta_h$ as $f_{\bbeta}(T_i^{(k)})=\eps d^{-(r+1/2)} \beta_h \psi_{j_n,h}(T_i^{(k)})$. Next we give sufficient conditions for Assumptions \ref{ass:W-expect-control} and \ref{ass:tail}.

\begin{assumption}(Assumptions to bound the likelihood ratio)\label{ass:expo_design}
\begin{enumerate}
    \item The log-likelihood ratio for the $i$th observation takes the form
    \begin{align*}
        W_{d,i}^{(k)} &= g_1(T_i^{(k)})h(X_i^{(k)}) - g_2(T_i^{(k)}),
    \end{align*}
    for some functions $g_1, g_2$ and $h$. 
    \item The conditional distribution $h(X_i^{(k)})|T_i^{(k)}$ follows a sub-exponential distribution (see Section \ref{sec:subexp} for definition), i.e.
$$h(X_i^{(k)})|T_i^{(k)} \sim SE\big(\nu^2(T_i^{(k)}), b(T_i^{(k)})\big).$$
    \item $W_{d,i}^{(k)}$ satisfies the following almost-sure bounds:
    \begin{align*}
        \max_{t \in [0,1]} \mathbb{E}[W_{d,i}^{(k)} \mid T_i^{(k)}=t] - \min_{t \in [0,1]} \mathbb{E}[W_{d,i}^{(k)} \mid T_i^{(k)}=t] &= O(d^{-r}), \\
        \max_{t \in [0,1]} \nu^2(t)g_1^2(t) &=  O(d^{-2r}), \\
        \max_{t \in [0,1]} g_1(t)b(t) &= O(d^{-r}).
    \end{align*}
    \item For some $c_1 > 0$ the expectation satisfy
    \begin{align*}
        \mathbb{E}[W_{d,i}^{(k)}] &=  O(d^{-2r}), \\
        c_1  d^{-1-2r}&\leq \mathbb{E}\left[\nu^2(T_i^{(k)})g_1^2(T_i^{(k)})\mathbbm{1}_{T_{i}^{(k)} \in I_{h}}\right]. 
    \end{align*}
\end{enumerate}
\end{assumption}
The assumption that $\mathbb{E}[W_{d,i}^{(k)}] = O(d^{-2r})$ implies Assumption \ref{ass:W-expect-control}. Finally, Assumption \ref{ass:expo_design} implies Assumption \ref{ass:tail}.
\begin{lem}\label{lem:lower-bound-sufficient-condition}
 For $m\leq n^{\gamma}$, with $\gamma<2r$, Assumption \ref{ass:expo_design} implies Assumption \ref{ass:tail}.
\end{lem}
The proof is deferred to Section \ref{sec:proof-sufficient:lower}.

\subsubsection{Sufficient conditions for upper bound} \label{sec:sufficient:UB}
For model \eqref{eq:model:reg}, we establish the upper bounds based on the assumptions below.

\begin{assumption}\label{ass:UB-suff-cond} 
Assume that for $j\leq j_n$ and $0\leq h\leq 2^j-1$,
\begin{align*}
    \hat{f}_{n,jh}^{(k)} = \frac{1}{n}\sum_{i=1}^n h(X_i^{(k)})\psi_{jh}(T_i^{(k)}),
\end{align*}
where the function $h(X)$ satisfies
\begin{enumerate}
    \item $\mathbb{E}[h(X) \mid T] = f(T)$,
    \item for $q = \lceil 3/r + 6 \rceil$, $\max_{t \in [0,1]}\mathbb{E}[|h(X)|^q \mid T=t] < \infty$.
\end{enumerate}
\end{assumption}
The first condition on $h$ is necessary to construct unbiased estimators, while the second is needed to establish a tail bound that ensures that the truncation of the estimators has an asymptotically negligible effect. For many models, we expect the density $p(x \mid f(t))$ to be continuous in $t$, in which case $\mathbb{E}[|h(x)|^q \mid T=t]$ is also continuous in $t$. Therefore the maximum for $t\in [0,1]$ being bounded is equivalent to the expectation being finite for any $t \in [0,1]$.

\begin{lem} \label{lem:upper-bound-assumption-consequences}
Assumption \ref{ass:UB-suff-cond} implies Assumption \ref{ass:UB}.
\end{lem}
The proof of the lemma is given in Section \ref{sec:proof-sufficient:upper}

\section{Examples}\label{sec:examples}
We apply our abstract theorem over various nonparametric models, including nonparametric regression, density estimation, binary regression, Poisson regression and heteroscedastic regression. We consider the case where each machine has the same budget $B=B^{(1)}=...=B^{(m)}$ and show that in each case the $L_2$-minimax estimation rate for $f\in B_{22}^r(L)$ is $\eps_{n,m,B}=\Nbass^{-r/(1+2r)}$, which can be rewritten (up to a logarithmic factor) as
\begin{align} \label{def:eps}
\eps_{n,m,B}&= \begin{cases} N^{-\frac{r}{1+2r}}, & B \geq N^{1/(1+2r)}/\log{N}, \\
\left(NB\log{N}\right)^{-\frac{r}{2+2r}}, & \left(n\log{N}\right)^{1/(1+2r)}/m \leq B < N^{1/(1+2r)}/\log{N}, \\
(n\log{N})^{-\frac{r}{1+2r}}, & B < \left(n\log{N}\right)^{1/(1+2r)}/m. \end{cases}
\end{align}

\subsection{Nonparametric regression}
Assume that at each machine $k$ we observe a sample $(T_1^{(k)},X_1^{(k)})$, $(T_2^{(k)},X_2^{(k)})$, ..., $(T_n^{(k)},X_n^{(k)})$, $k=1,...,m$, satisfying
\begin{align}
    X_i^{(k)} \mid T_i^{(k)} &\overset{ind}{\sim} N(f(T_i^{(k)}), 1),\quad T_i^{(k)} \overset{iid}{\sim} \text{Unif}(0,1),\quad i=1,...,n,\label{def:regression}
\end{align}
and our goal is to recover the underlying functional parameter $f\in B_{22}^r(L)$ of interest.

Minimax lower bounds and matching upper bounds for this model were already derived in \cite{szabo_adaptive_2019}. We re-derive the same bounds using our general results. This is a fairly simple model to analyze because the likelihood ratio $W_d^{(k)}$ itself has a Gaussian distribution whose expectation and tails are easily controlled. The minimax rate for this model is given in the following proposition and the corresponding proof is deferred to Section \ref{sec: proof:prop:regression}.

\begin{prop} \label{prop:nonparametric-regression}
The minimax estimation rate in the distributed nonparametric regression model  \eqref{def:regression} over the class of distributed estimators $\mathcal{F}_{\text{dist}}(\boldsymbol{B})$ is bounded from below by $ \eps_{n,m,B}$, given in $\eqref{def:eps}$. This minimax rate is achieved for $m\leq n^{\gamma}$, $\gamma<2r$, by the distributed estimator $\hat{f}_{n,m}$, given in \eqref{eq: block:est}, up to a logarithmic factor, i.e.
\begin{align*}
   \sup_{f\in B_{22}^r(L)}\mathbb{E}_{f} \|\hat{f}_{n,m}-f\|_2\leq \log (N) \eps_{n,m,B}.
\end{align*}
\end{prop}

\subsection{Density estimation}
In the distributed density estimation problem in each machine $k=1,...,m$ we observe an iid sample 
\begin{align}
X_i^{(k)} \stackrel{iid}{\sim} f,\quad i=1,...,n,\label{def:density}
\end{align}
 from an unknown density $f\in B_{22}^r(L)$ of interest. This model, unlike the other examples, can not be formulated as in \eqref{eq:model:reg} in a natural way. Therefore, in the proof of the proposition, deferred to Section \ref{sec: proof:prop:density}, we verify the original Assumptions \ref{ass:W-expect-control}-\ref{ass:tail}.

\begin{prop} \label{prop:density-estimation}
The minimax estimation rate in the distributed density estimation model \eqref{def:density} over the class of distributed estimators $\mathcal{F}_{\text{dist}}(\boldsymbol{B})$ is bounded from below by $ \eps_{n,m,B}$, given in $\eqref{def:eps}$ and the estimator $\hat{f}_{n,m}$ given in \eqref{eq: block:est} achieves this optimal rate up to a logarithmic factor, i.e.
\begin{align*}
   \sup_{f\in B_{22}^r(L)}\mathbb{E}_{f} \|\hat{f}_{n,m}-f\|_2\leq \log (N) \eps_{n,m,B}.
\end{align*}
\end{prop}

\subsection{Binary regression}
In the distributed binary regression model each machine $k=1,...,m$, obtains the binary response $X_i^{(k)} \in \{0,1\}$, $i=1,...,n$ with a generative model of the form
\begin{align}
 X_i^{(k)} \mid T_i^{(k)} \overset{ind}{\sim} \text{Bern}(f(T_i^{(k)})),\quad   T_i^{(k)} &\overset{iid}{\sim} \text{Unif}(0,1),\quad i=1,...,n, \label{def:binary}
\end{align}
for $f\in B_{22}^r(L)$, $0\leq f\leq 1$. 

This model is flexible enough to analyze binary regression models with common link functions, such as logistic and probit regression. If $\psi$ is a smooth function, then the composition $\psi(f(T_i^{(k)}))$ shares the smoothness of $f$. Therefore to analyze the model $\mathbb{P}(X_i^{(k)} = 1 \mid T_i^{(k)}) = \psi(f(T_i^{(k)}))$ we may simply consider $\mathbb{P}(X_i^{(k)} = 1 \mid T_i^{(k)}) = f(T_i^{(k)})$, slightly abusing our notations. The minimax properties and adaptive estimators for binary regression in the classical non-distributed setting are studied for instance in \cite{adaptivebinaryreg}. Binary regression is an example of a model where the log-likelihood-ratios $W_{d}^{(k)}$ are bounded. The minimax behavior is summarized in the proposition below and its proof is deferred to Section \ref{sec: proof:prop:binary}.

\begin{prop} \label{prop:binary-regression}
The distributed minimax estimation rate in the distributed binary regression model \eqref{def:binary} is bounded from below by $\eqref{def:eps}$ and for $m\leq n^{\gamma}$, with $\gamma<2r$, the estimator $\hat{f}_{n,m}$ given in \eqref{eq: block:est} achieves this optimal rate up to a logarithmic factor, i.e.
\begin{align*}
   \sup_{f\in B_{22}^r(L)}\mathbb{E}_{f} \|\hat{f}_{n,m}-f\|_2\leq \log (N) \eps_{n,m,B}.
\end{align*}
\end{prop}

\subsection{Poisson regression}
In the distributed version of the nonparametric Poisson regression model, each machine $k\in\{1,...,m\}$ receives an iid sample $(T_1^{(k)},X_1^{(k)}),...,(T_n^{(k)},X_n^{(k)})$ satisfying
\begin{align}
    X_i^{(k)} \mid T_i^{(k)}\stackrel{ind}{\sim} \text{Pois}(f(T_i^{(k)})),\quad T_i^{(k)} \stackrel{iid}{\sim} \text{Unif}(0,1),\quad i=1,...,n,\label{def:poisson}
\end{align}
with $f\in B_{22}^r(L)$, $f>0$.

In this case the log-likelihood-ratio $W_{d}^{(k)}$ is unbounded and has sub-exponential tails.  This requires a more careful analysis compared to the benchmark nonparametric regression model with sub-Gaussian tails for the log-likelihood-ratio, see Section \ref{sec: proof:prop:poisson} for the proof of the proposition below.

\begin{prop} \label{prop:poisson-regression}
The distributed minimax estimation rate in the distributed Poisson regression model \eqref{def:poisson} is bounded from below by $\eqref{def:eps}$ and for $m\leq n^{\gamma}$, with $\gamma<2r$, the estimator $\hat{f}_{n,m}$ given in \eqref{eq: block:est} achieves this optimal rate up to a logarithmic factor, i.e.
\begin{align*}
   \sup_{f\in B_{22}^r(L)}\mathbb{E}_{f} \|\hat{f}_{n,m}-f\|_2\leq \log (N) \eps_{n,m,B}.
\end{align*}
\end{prop}

\subsection{Heteroskedastic regression}
Finally, we consider a distributed version of the heteroskedastic regression model where in each machine $k\in\{1,...,m\}$ we observe iid pairs of random variables $(T_1^{(k)},X_1^{(k)})$, ....,$(T_n^{(k)},X_n^{(k)})$ satisfying
\begin{align}
    X_i^{(k)} \mid T_i^{(k)} &\stackrel{ind}{\sim} N(0, f(T_i^{(k)})), \quad T_i^{(k)} \stackrel{iid}{\sim} \text{Unif}(0,1),\quad i=1,...,n,\label{def:heteroscedastic}
\end{align}
for some unknown, non-negative functional parameter $f\in B_{22}^r(L)$.

This model also falls outside of the standard models as the corresponding log-likelihood ratio $W_{d}^{(k)}$ is unbounded with sub-exponential tails and like the Poisson regression model requires a more careful analysis, see Section \ref{sec: proof:prop:heteroskedastic}.

\begin{prop} \label{prop:heteroskedastic-regression}
The distributed minimax estimation rate in the distributed version of the heteroscedastic regression model \eqref{def:heteroscedastic} is bounded from below by $ \eps_{n,m,B}$, given in $\eqref{def:eps}$. For $m\leq n^{\gamma}$, with $\gamma<2r$, the distributed estimator $\hat{f}_{n,m}$ given in \eqref{eq: block:est} achieves this optimal rate up to a logarithmic factor, i.e.
\begin{align*}
   \sup_{f\in B_{22}^r(L)}\mathbb{E}_{f} \|\hat{f}_{n,m}-f\|_2\leq \log (N) \eps_{n,m,B}.
\end{align*}
\end{prop}

\section{Concluding remarks and future work}\label{sec:discuss}
In this paper we have studied the problem of nonparametric estimation in a distributed setting under communication constraints. We introduced the bit adjusted sample size (BASS), which characterizes the efficient information content in a distributed setting taking the communication constraint into account. We have identified sufficient conditions on the model for deriving distributed minimax lower bounds and a sufficient condition on a block distributed estimator achieving the theoretical lower bounds (up to a logarithmic factor). The abstract results were applied in a variety of nonparametric models including nonparametric regression, density estimation, classification, Poisson regression and volatility estimation. 

Extending our results to high-dimensional models with sparse structure is a natural and interesting direction, which however, requires new ideas due to different underlying model structure. A possible future direction is to derive adaptive distributed procedures for the unknown regularity of the underlying functional parameter of interest achieving the minimax lower bounds. Such kind of estimators were considered in context of the nonparametric regression model based on the distributed Lepskii's method in \cite{szabo_adaptive_2019}. Another potential direction is to derive adaptive estimators which minimize communication between machines while achieving the non-distributed minimax rate; see \cite{szabo2020distributed,cai:2021:distributed} for such approaches in context of the nonparametric regression and the Gaussian white noise models.

\section{Proofs of the minimax distributed lower and upper bounds}\label{sec:proofs}

\subsection{Proof of Theorem \ref{thm:gen_lower_bound_thm}} \label{sec:proof-lower-bound}

The proof of the theorem can be viewed as the extension of Theorem 2.1 of \cite{szabo_adaptive_2019} to a more general, abstract framework. Here, however, we do not carry out explicit computations tailored to the regression model, but keep the setting more general, allowing to apply our results to a variety of models.

We first put a prior on our testing sieve; let $F \in \{-1,1\}^d$ be a uniform random element on the discrete hypercube independent of $T$. This indexes a uniform distribution over the testing sieve by the bijection $F = \bbeta \leftrightarrow \fbb$. First we give an upper bound for the mutual information $I(F;Y)$ between the prior on the sieve $F$ and the transmitted information $Y$, which will be needed to apply Theorem \ref{thm:fanos} (a variant of Fano's inequality).

\begin{lem} \label{lem:ass1-info-bound}
Under Assumption \ref{ass:W-expect-control}, we have $I(F;Y^{(k)}) \leq C\epsilon^2 n d^{-2r}$ for all $k = 1,\ldots, m$.
\end{lem}
\begin{proof}
Applying the data processing inequality to the Markov chain $F \to X^{(k)} \to Y^{(k)}$ we get that $I(F; Y^{(k)}) \leq I(F; X^{(k)})$. From the definition of the mutual information, the convexity of the KL-divergence, the independence of the observations and finally the definition of $W_{d}^{(k)}$ together with Assumption \ref{ass:W-expect-control} we get that
\begin{align*}
    I(F; Y^{(k)}) &\leq \sum_{\bbeta} \frac{1}{2^d} KL\left(p_f(X^{(k)} \mid F = \bbeta),  \frac{1}{2^d}\sum_{\bbeta'} p(X^{(k)}|F= \bbeta')\right) \\
       &\leq \frac{1}{2^{2d}}  \sum_{\bbeta,\bbeta'}KL\left(p(X^{(k)} \mid F = \bbeta), p(X^{(k)} \mid  F = \bbeta')\right)\\
       &=\frac{1}{2^{2d}}  \sum_{\bbeta,\bbeta'}\sum_{i=1}^n KL\left(p(X_i^{(k)} \mid F = \bbeta), p(X_i^{(k)} \mid  F = \bbeta')\right)\\
&= \frac{1}{2^{2d}}  \sum_{\bbeta,\bbeta'} \sum_{i=1}^n \mathbb{E}_{\bbeta}  W_{d}^{(k)}( X_i^{(k)}, \bbeta, \bbeta') \leq C\epsilon^2 nd^{-2r}.
\end{align*}
\end{proof}

The next bound on the mutual information following from Assumption \ref{ass:tail} relates the communication constraint to the entropy of the binary string $H(Y^{(k)})$, which can be related to the length of the code $l(Y^{(k)})$ and therefore to the communication budget via a modified version of Shannon's source coding theorem. 

\begin{lem} \label{lem:ass2-info-bound}
Under Assumption \ref{ass:tail} with $d \asymp \Nbass^{1/(1+2r)}$ we have for some $C>0$ that
\begin{align*}
    I(F; Y^{(k)}) &\leq C\epsilon^2d^{-1-2r} n\log(N)H(Y^{(k)}) + o(m^{-1}).
\end{align*}
\end{lem}
\begin{proof}
Let us start by defining the latent variables $U_i^{(k)}$, $i=1,...,n$, $k=1,...,m,$ such that  $U_i^{(k)}=h$ if $X_i^{(k)}\in \chi_h$. Note that $U_i^{(k)}$s are independent, discrete uniformly distributed on the set $[d]=\{1,...,d\}$, and independent from $F$. Hence, in view of the chain rule and the nonnegativity of the mutual information we get that
\begin{align*}
I(F;Y^{(k)})&\leq I(F;Y^{(k)})+  I(F;U^{(k)}|Y^{(k)} )=I\big(F; (Y^{(k)},U^{(k)}) \big)\\
& = I(F;Y^{(k)}|U^{(k)} )+ I(F;U^{(k)})=I(F;Y^{(k)}|U^{(k)} )\\
&=\sum_{\boldsymbol{u}\in [d]^n} I(F; Y^{(k)}| U^{(k)}=\boldsymbol{u})P(U^{(k)}=\boldsymbol{u}).
\end{align*}
Furthermore, in view of the assumption that $X^{(k)}_i| (U^{(k)}_i=u_i ,F_{u_i})$ is independent of $F_{u_i'}$, $u_i'\neq u_i$,  we can apply Theorem A.13 of \cite{szabo_adaptive_2019}, implying that
\begin{align}
    I(F; Y^{(k)}|U^{(k)}=\boldsymbol{u}) &\leq \sum_{h=1}^d \Big[ (\log{2})\sqrt{\mathbb{P}_{\beta_h}(E_h^{(k)}(a_n)^\complement | U^{(k)}=\boldsymbol{u})}\nonumber \\
&  \quad+ (\log{|\mathcal{F}_0(r,\epsilon)|})\mathbb{P}_{\beta_h}(E_h^{(k)}(a_n)^\complement| U^{(k)}=\boldsymbol{u} )\Big] \nonumber \\
    &\quad + 2e^{2a_n}(e^{a_n} -1)^2I(X^{(k)}; Y^{(k)}| U^{(k)}=\boldsymbol{u}),\label{eq: UB:cond:mut:info}
\end{align}
where the event $E_h^{(k)}(a_n)$ was defined in Assumption \ref{ass:tail}.

Noting that $\log{|\mathcal{F}_0(d,C_0,\epsilon)|}\leq d\log 2$, we have in view of Assumption \ref{ass:tail}, Jensen's inequality and that the mutual information with $Y^{(k)}$ is always bounded by the entropy $I(X^{(k)}; Y^{(k)}| U^{(k)}) \leq H(Y^{(k)})$,
\begin{align*}
 I(F; Y^{(k)}|U^{(k)}) &\leq \log (2) \sum_{h=1}^d \Big[\sqrt{\mathbb{P}_{\beta_h}(E_h^{(k)}(a_n)^\complement)}+d \mathbb{P}_{\beta_h}(E_h^{(k)}(a_n)^\complement) \Big]\\
&\qquad + 2e^{2a_n}(e^{a_n} -1)^2    I(X^{(k)}; Y^{(k)}| U^{(k)})\\
&= o(m^{-1})+ 2e^{2a_n}(e^{a_n} -1)^2 H(Y^{(k)}).
\end{align*}
Since $d \asymp \Nbass^{1/(1+2r)}$,  by Proposition \ref{cor:Nbass-symmetric} we get that $d^{2r+1} \geq cn\log{N}$ and thus 
\begin{align*}
    a_n^2 = C^2\epsilon^2 d^{-1-2r}n\log{N} \leq C'\eps^2,
\end{align*}
for some constant $C'>0$. Furthermore, by the convexity of the function $x\mapsto e^x$, there exists a constant $C$ such that $e^{x} \leq 1 + Cx$ for all $x \in [0, C']$. Putting everything together we get that
\begin{align*}
    I(F; Y^{(k)}) &\leq Ca_n^2H(Y^{(k)}) + o(m^{-1}) = C'\epsilon^2 d^{-1-2r}n\log(N)H(Y^{(k)}) + o(m^{-1}),
\end{align*}
completing the proof. 
\end{proof}

Finally, we note that  one can adapt assertion (3.4) in \cite{szabo_adaptive_2019} to the present setting. For completeness we provide the proof of the lemma in the supplement. 

\begin{lem}\label{lem: help:Minimax:LB}
For $\Nbass^{1/(1+2r)}\geq 20$ we get that
\begin{align*}
      \inf_{\hat{f}_{n,m} \in \mathcal{F}_{\text{dist}}(\boldsymbol{B})} \sup_{f \in B_{2,2}^r(L) } \mathbb{E}_{f}\Vert \hat{f}_{n,m} - f\Vert_2^2 &\gtrsim \Nbass^{-2r/(1+2r)} \left(1 - \frac{I(F;Y) + \log{2}}{\Nbass^{1/(1+2r)}/6}\right).
\end{align*}
\end{lem}

We are now ready to prove Theorem \ref{thm:gen_lower_bound_thm}. Based on Lemma \ref{lem: help:Minimax:LB}, it is sufficient to show that $I(F, Y) < \Nbass^{1/(1+2r)}/7$. In view of $d\asymp  \Nbass^{1/(1+2r)}$, and Lemma \ref{lem:ass1-info-bound} and \ref{lem:ass2-info-bound},
\begin{align*}
    I(F;Y) &\leq \sum_{k=1}^m I(F; Y^{(k)}) \\
&\leq \sum_{k=1}^m \min\left(Cn\epsilon^2\Nbass^{-2r/(1+2r)}, C'\Nbass^{-1}\epsilon^2 n\log(N)H(Y^{(k)}) + o(m^{-1})\right) \\
    &\leq \epsilon^2 n\Nbass^{-2r/(1+2r)}\sum_{k=1}^m \min\left(1, C'\frac{\log{N}}{\Nbass^{1/(1+2r)}}H(Y^{(k)})\right) + o(1).
\end{align*}
Furthermore, from Lemma 5.3 of \cite{szabo_adaptive_2019} we know that 
$$H(Y^{(k)}) \leq 2\mathbb{E}[l(Y^{(k)})] + 1\leq 2 B^{(k)}+1,$$
so
\begin{align*}
    n \sum_{k=1}^m \min\Big(1, C'\frac{\log{N}}{\Nbass^{1/(1+2r)}}H(Y^{(k)})\Big) \lesssim n\sum_{k=1}^m \min\Big(1, \frac{\log{N}}{\Nbass^{1/(1+2r)}}B^{(k)}\Big) \leq  \Nbass,
\end{align*}
which in turn implies that $ I(F; Y)\lesssim \epsilon^2\Nbass^{1/(1+2r)}$. Picking $\epsilon>0$ small enough concludes the proof.

%
%

\subsection{Proof of Theorem \ref{thm:gen_upper_bound_thm}} \label{sec:proof-of-upper-bound}

First we show that $\hat{f}_{n,m}\in\mathcal{F}_{dist}(B)$. On machine $k$, consider $l(Y_{jh}^{(k)})$, which is one if $\vert\hat{f}_{jh,n}^{(k)}\vert \geq \sqrt{N}$ and $\log_2{|\hat{f}_{jh,n}^{(k)}|} + (1/2)\log_2{N} \leq \log_2{N}$ if $\hat{f}_{jh,n}^{(k)} < \sqrt{N}$. Therefore $l(Y_{jh}^{(k)}) \leq \log_2{N}$ almost surely. The number of estimated coefficients is $\lfloor B/\log_2{N}\rfloor$, so the number of transmitted bits is bounded above by $\log_2{N} \times \lfloor B/\log_2{N}\rfloor \leq B$ bits on each machine. Thus the communication budget is always satisfied. 

Next we show that the estimator achieves the minimax distributed rate (up to a logarithmic factor). Let $R_n$ denote the event that none of the transmitted strings $Y_{jh}^{(k)}$, for $ j\in\{0,..., j_n\}, h\in\{0,..., 2^j-1\}$ and $k\in\{1,...,m\}$, is equal to zero. Then 
 \begin{align}
       f_{jh} - \hat{f}_{jh} = f_{jh} - | U_{jh}|^{-1}\sum_{k \in U_{jh}} \tilde{f}_{jh}^{(k)}= Z_{jh} + \epsilon_{jh},\label{eq: bias}
\end{align}
where 
$\epsilon_{jh} = | U_{jh}|^{-1}\sum_{k \in U_{jh}}(\hat{f}_{n,jh}^{(k)} - \tilde{f}_{jh}^{(k)}) $, $Z_{jh} = | U_{jh}|^{-1}\sum_{k \in U_{jh}} (f_{jh} - \hat{f}_{n,jh}^{(k)})$ and $U_{jh}=\{ (\ell-1)\kappa+1,....,\ell\kappa\}$.  Note that by Assumption \ref{ass:UB} on the event $R_n$ we have $\epsilon_{jh}\leq | U_{jh}|^{-1}\sum_{k \in U_{jh}}1/\sqrt{N}\leq 1/\sqrt{N}$. Furthermore, by the same assumption, $\mathbb{E}[Z_{jh}] =0$  and $\mathbb{E}[Z_{jh}^2]=\text{var}(Z_{jh} ) = | U_{jh}|^{-2} \sum_{k \in U_{jh}} \text{var}( \hat{f}_{n,jh}^{(k)})\leq \sigma^2/(\kappa n)$.

Next we decompose the mean integrated squared error into two terms
    \begin{align*}
        \mathbb{E}_{f}\Vert f - \hat{f}_{n,m} \Vert_2^2 &= \mathbb{E}_{f}\Vert \hat{f}_{n,m} - f \Vert_2^2I_{R_n} +\mathbb{E}_{f}\Vert \hat{f}_{n,m} - f \Vert_2^2I_{R_n^\complement}
    \end{align*}
and deal with them separately.  First note that in view of \eqref{eq: bias}, and the inequality $(a+b)^2\leq 2a^2+2b^2$,
    \begin{align*}
        \mathbb{E}_{f}\Vert \hat{f}_{n,m} - f \Vert_2^2I_{R_n} &= \sum_{j=0}^\infty\sum_{h=0}^{2^j-1} \mathbb{E}_{f}(f_{jh} - \hat{f}_{jh})^2I_{R_n} \\
           &\leq 2\sum_{j=0}^{j_n}\sum_{h=0}^{2^j-1}\big( \mathbb{E}_{f}(\epsilon_{jh}^2I_{R_n}) + \mathbb{E}_{f}(Z_{jh}^2)\big) +\sum_{j > j_n}\sum_{h=0}^{2^j-1} f_{jh}^2  \\
        &\lesssim 2^{j_n + 1}\big(1/N+\sigma^2/(n\kappa) \big) + 2^{-2r j_n}L \lesssim \Nbass^{-\frac{2r}{1+2r}}, 
    \end{align*}
    since $\Nbass/2 < n\kappa \leq \Nbass\leq N$. Furthermore,
    \begin{align*}
        \mathbb{E}_{f}\left[\Vert \hat{f}_{n,m} - f \Vert_2^2I_{R_n^\complement}\right] &\leq 2\mathbb{E}_{f}\left[\Vert \hat{f}_{n,m} \Vert_2^2I_{R_n^\complement}\right] + 2\Vert f \Vert_2^2\mathbb{P}(R_n^\complement).
    \end{align*}
Then note that $\Vert \hat{f}_{n,m} \Vert_2^2\leq \sum_{j=0}^{j_n}\sum_{h=0}^{2^j-1}\hat{f}_{jh}^2 {\leq 2^{j_n+1}N}\leq 2N^{\frac{2+2r}{1+2r}}$ and by the union bound,
\begin{align*}
\mathbb{P}(R_n^\complement) \leq   \sum_{k=1}^m
\sum_{j=0}^{j_n}\sum_{h=0}^{2^j-1}\mathbb{P}(\hat{f}_{n,jh}^{(k)} > \sqrt{N})\lesssim m2^{j_n}o(m^{-1}N^{-\frac{3+4r}{1+2r}}) = o(N^{-2}). 
\end{align*}
By combining the above bounds we get that
    \begin{align*}
        \mathbb{E}_{f}\left[\Vert \hat{f}_{n,m} - f \Vert_2^2I_{R_n^\complement}\right] &\lesssim N^{\frac{2 + 2r}{1 + 2r}}\mathbb{P}(R_n^\complement)= o(N^{-\frac{2r}{1+2r}}).
    \end{align*}
    We conclude the proof by combining the above upper bounds on the events $R_n$ and $R_n^\complement$,
    \begin{align*}
       \mathbb{E}_{f}\Vert \hat{f}_{n,m}  - f \Vert_2^2&\lesssim\Nbass^{-\frac{2r}{1+2r}}+ o(N^{-\frac{2r}{1+2r}})\lesssim \Nbass^{-\frac{2r}{1+2r}}.
    \end{align*}

\section{Proof of sufficient conditions for regression type models}\label{sec:proof-sufficient}
In this section we collect the proofs for the sufficient conditions of our main assumptions.

\subsection{Proof of Lemma \ref{lem:lower-bound-sufficient-condition}}\label{sec:proof-sufficient:lower}
By factoring the joint density of $(X,T)$ as the density of $X \mid T$ multiplied by the density of $T$ and noting that the density of $T$ does not depend on $\beta$, we see that the likelihood ratio $W_{d,i}^{(k)}$ is the likelihood ratio of the conditional distributions of $X \mid T$. First note that $W_{d,i}^{(k)}(X_i^{(k)},T_i^{(k)}, \beta, \beta') |(T^{(k)}_i \in I_h)$ only depends on $\beta_h$ and $\beta_h'$. Therefore, by slightly abusing our notations we write that
\begin{align*}
    W_{[h]}^{(k)}( \beta_h, \beta_h')=  \sum_{i=1}^n W_{d,i}^{(k)}(X_i^{(k)}, T_i^{(k)}, \beta_h, \beta_h')\mathbbm{1}_{T_i^{(k)} \in I_h}.
\end{align*}
The event $E_h^{(k)}$ in Assumption \ref{ass:tail} concerns the maximum over different combinations of $\beta_h$ and $\beta_h'$. As there are only four possible combinations of these parameters, the union bound shows that the asymptotic results are not affected if we show the tail bound only for fixed $\beta_h$ and $\beta_h'$ rather than the maximum. Finally, let us denote by $\mathcal{T}_h^{(k)}$ the set of indexes for which the corresponding observations belong to the $h$th bin, i.e. $i\in \mathcal{T}_h^{(k)}$ if $T_i^{(k)}\in I_h$.

Note that in view of Lemma \ref{lem:prop-of-sub-expo-rvs} below, $W_{d,i}^{(k)}|T_i^{(k)} \sim SE(V(T_i^{(k)}),g_1(T_i^{(k)})b(T_i^{(k)}))$, with $V(T_i^{(k)}) = \nu^2(T_i^{(k)})g_1^2(T_i^{(k)})$. Furthermore, let us define the events

\begin{align*}
\mathcal{C}_1^{(k)} &= \left[ \frac{n}{2d} \leq |\mathcal{T}_h^{(k)}|\leq \frac{2n}{d}, h=1,\ldots,d \right], \\
    \mathcal{C}_{2h}^{(k)} &= \left[\mathbb{E}[W_{[h]}^{(k)} \mid T^{(k)}] \leq C\eps d^{-1/2-r}\sqrt{n\log N}\right], \\
    \mathcal{C}_{3h}^{(k)} &= \Bigg[\Bigg\vert \sum_{i=1}^n \Big(V(T_i^{(k)})\mathbbm{1}_{T_i^{(k)}\in I_h} - \mathbb{E}_{T^{(k)}}V(T_i^{(k)}) \mathbbm{1}_{T_i^{(k)}\in I_h}\Big)\Bigg\vert \leq C\epsilon^2d^{-2r-1/2}\sqrt{n\log{N}}\Bigg],
\end{align*}
 and their intersection as
\begin{align}
    \mathcal{C}^{(k)} &= \mathcal{C}_1^{(k)} \cap \left[\bigcap_{h=1}^d\mathcal{C}_{2h}^{(k)}\right] \cap \left[\bigcap_{h=1}^d\mathcal{C}_{3h}^{(k)}\right]. \label{eq:define-C_k-event}
\end{align}

Then in view of Lemma \ref{lem:C_h_combine_bound} below, $\mathbb{P}\big((\mathcal{C}^{(k)})^\complement\big)=o(m^{-2}d^{-2})$ holds. 
Therefore by applying Lemma \ref{lem:sub_exp_cond_prob_bound} below with $C'> \eps^{-2}(2 + \frac{2}{1 + 2r})$ we get
\begin{align*}
      \mathbb{P}_{\bbeta}\left[E_h^{(k)}(a_n)^\complement \right]
&\leq \int_{t\in \mathcal{C}^{(k)}}\mathbb{P}_{\bbeta}\left[E_h^{(k)}(a_n)^\complement | T^{(k)}=t\right]p_{T}(t)dt+\mathbb{P}\big((\mathcal{C}^{(k)})^\complement\big)\\
&\leq 4e^{-C'\eps^2 \log{N}}+ o(m^{-2}d^{-2})=o(m^{-2}d^{-2}),
\end{align*}
concluding our proof.

The next lemma shows that under Assumption \ref{ass:expo_design} the complement of $ \mathcal{C}^{(k)}$ has (asymptotically) negligible probability.

\begin{lem}[Bound on $\mathcal{C}^{(k)}$] \label{lem:C_h_combine_bound}
Under Assumption \ref{ass:expo_design} and $m\leq n^{\gamma}$, for some $\gamma<2r$, it follows that $\mathbb{P}\big((\mathcal{C}^{(k)})^{\complement}\big)=o(m^{-2}d^{-2})$. 
\end{lem}
\begin{proof}
By triangle inequality
\begin{align}
    \mathbb{P}\big((\mathcal{C}^{(k)})^{\complement}\big) &\leq \mathbb{P}\big((\mathcal{C}_1^{(k)})^\complement\big) + \sum_{h=1}^d \mathbb{P}\big((\mathcal{C}_{2h}^{(k)})^\complement|\mathcal{C}_1^{(k)})\mathbb{P}(\mathcal{C}_1^{(k)}) + \sum_{h=1}^d \mathbb{P}\big((\mathcal{C}_{3h}^{(k)})^\complement|\mathcal{C}_1^{(k)}\big)\mathbb{P}(\mathcal{C}_1^{(k)})\nonumber \\
&\leq 2d\exp(-n/(8d)) + d\mathbb{P}\big((\mathcal{C}_{21}^{(k)})^\complement \mid \mathcal{C}_1^{(k)}\big) + d\mathbb{P}(\mathcal{C}_{31}^{(k)} \mid \mathcal{C}_1^{(k)}),\label{eq:UB:comp:event}
\end{align}
where the last line follows by applying Chernoff's bound, see Lemma 5.2 of \cite{szabo_adaptive_2019} and the symmetry in $h=1,...,d$. We show below that each term on the right hand side multiplied by $(md)^2$ tends to zero.

The inequalities $m \leq n^\gamma$, $\gamma<2r$ and $d\asymp N_{bass}^{1/(1+2r)}  \leq N^{1/(1+2r)}$ imply $d \leq n^{(\gamma + 1)/(1 + 2r)}$ and thus
\begin{align*}
    m^2d^3\exp(-n/(8d)) &\leq \exp\left(-n^{(2r - \gamma)/(1+2r)}/8 + \left[\gamma + \frac{3\gamma + 3}{1 + 2r}\right]\log{n}\right) =o(1).
\end{align*}

We will use Lemma \ref{lem:gen-prob-lemma} to bound $\mathbb{P}((\mathcal{C}_{21}^{(k)})^\complement \mid \mathcal{C}_1^{(k)})$. The iid sequence $X_1,\ldots, X_n$ from Lemma \ref{lem:gen-prob-lemma} will be $T_1^{(k)},\ldots, T_n^{(k)}$ with $h(T_i^{(k)}) = \mathbb{E}[W_{d,i}^{(k)} \mid T_i^{(k)}]$. With this choice we have $S_n = \sum_{i=1}^n h(T_i^{(k)})\mathbbm{1}_{T_i^{(k)} \in I_h} = \mathbb{E}[W_{[h]}^{(k)} \mid T^{(k)}]$. Letting $\delta_n = d^{-r}$ and $\chi_n = I_h$, we have from Assumption \ref{ass:expo_design} that the support condition on $h$ is satisfied with $\alpha = 1$. For the expectation, we recall that $\mathbb{P}(T_i^{(k)} \in I_h) = 1/d$ and $\mathbb{E}[W_{d,i}^{(k)} \mid T_i^{(k)}  \in I_h]$ is equal for all $h$ and  conclude that 
\begin{align*}
\mathbb{E}[h(T_i^{(k)}) \mid T_i^{(k)}  \in I_h] = \sum_{h'=1}^d \mathbb{E}[W_{d,i}^{(k)} \mid T_i^{(k)} \in I_{h'}]\mathbb{P}(T_i^{(k)}  \in I_{h'}) = \mathbb{E}[W_{d,i}^{(k)}] = O(d^{-2r}).
\end{align*}
The set $I_n = \mathcal{T}_h^{(k)}$ satisfies the size constraint when $r_n = n/d$ based on the definition of $\mathcal{C}_1^{(k)}$. Note that we have $d^{-r}\sqrt{r_{n}} <1$ since $d^{2r+1}\gtrsim N_{bass}\gg n$. Therefore, for large enough $C>0$ in the definition of $\mathcal{C}_{21}^{(k)}$ we have for some $C'>2\gamma+3(1+\gamma)/(1+2r)$ that
\begin{align*}
m^2d^3\mathbb{P}\big((\mathcal{C}_{21}^{(k)})^\complement \mid \mathcal{C}_1^{(k)}\big) \leq  m^2d^3n^{-C'}=n^{2\gamma+3(1+\gamma)/(1+2r)-C'}= o(1).
\end{align*}

Finally, in view of Assumption \ref{ass:expo_design}, the support length of $V(T_i^{(k)})$ is bounded by $Cd^{-2r}$, hence by Hoeffding's inequality
\begin{align*}
   \mathbb{P}&\left[ \left\vert \sum_{i} V(T_i^{(k)})\mathbbm{1}_{T_i^{(k)} \in I_h } - \mathbb{E}\left(\sum_{i }V(T_i^{(k)})\mathbbm{1}_{T_i^{(k)} \in I_h }\right)\right\vert > C\epsilon^2d^{-2r}\sqrt{\frac{n\log{N}}{d}}\right]\\ 
&\qquad\qquad\qquad\leq \exp\left(-\frac{C^2\eps^4d^{-4r} n\log(N)/d}{d^{-4r}(2n/d)}\right) = \exp\left(-(C^2\eps^4/2)\log{N}\right),
\end{align*}
which by the same arguments as above shows that the third term on the right hand side of \eqref{eq:UB:comp:event} is of $o(m^{-2}d^{-2})$, finishing the proof of the lemma.
\end{proof}

Now we move on to the conditions controlling $\mathbb{P}(W_{[h]}^{(k)} \geq a_n \mid \mathcal{C}^{(k)})$ under the sub-exponential assumption on $h(X)$.

\begin{lem}[Sub-exponential bound on conditional probability] \label{lem:sub_exp_cond_prob_bound}
Under Assumption \ref{ass:expo_design} for $a_n = C_a \eps d^{-r-1/2}\sqrt{ n\log(N)}$ we have
\begin{align*}
    \mathbb{P}(W_{[h]}^{(k)} \geq a_n \mid T^{(k)}=t ) \leq e^{-C'\eps^2\log{N}},
\end{align*}
for all $ t \in \mathcal{C}^{(k)}$ and $k\in\{1,...,m\}$, where $\mathcal{C}^{(k)}$ is defined in \eqref{eq:define-C_k-event}, $W_{[h]}^{(k)}=W_{[h]}^{(k)}(\beta_h,\beta_h')$ for arbitrary $\beta_h,\beta_h'\in\{-1,1\}$, and  $C'$ can be made arbitrarily large by selecting $C_a$ large enough. 
\end{lem}
\begin{proof}
First note, that $W_{[h]}^{(k)}(1,1) = W_{[h]}^{(k)}(-1,-1) = 0$, for which the statement trivially holds, hence it remains to deal with the cases $W_{[h]}^{(k)}(1,-1)$ and $W_{[h]}^{(k)}(-1,1)$. Since by assumption $h(X_i^{(k)})|T_i^{(k)}\sim SE\big(\nu^2(T_i^{(k)}), b(T_i^{(k)}) \big)$, we have in view of Lemma \ref{lem:prop-of-sub-expo-rvs} that
\begin{align*}
 &W_{[h]}^{(k)}\mid T^{(k)} \sim SE\left(\sum_{ i \in \mathcal{T}_h^{(k)}} V(T_i^{(k)}), \max_{ i \in \mathcal{T}_h^{(k)}} b(T_i^{(k)})g_1(T_i^{(k)})\right),\\
&\mathbb{E}[W_{[h]}^{(k)} \mid T^{(k)}] = \sum_{i \in \mathcal{T}_h^{(k)}} \hspace{-7pt} g_1(T_i^{(k)})\mathbb{E}[h(X_i^{(k)}) \mid T_i^{(k)}] - g_2(T_i^{(k)}).
\end{align*}

We show below that 
\begin{align}
    0 \leq a_n - \mathbb{E}[W_{[h]}^{(k)} \mid T^{(k)}] &\leq \frac{\sum_{i \in \mathcal{T}_h^{(k)}} V(T_i^{(k)})}{\max_{i \in \mathcal{T}_h^{(k)}} b(T_i^{(k)})g_1(T_i^{(k)})},\label{cond1:lem10}\\
 \frac{(a_n - \mathbb{E}[W_{[h]}^{(k)} \mid T^{(k)}])^2}{\sum_{i \in \mathcal{T}_h^{(k)}} V(T_i^{(k)})} &\geq C'\log{N},\label{cond2:lem10}
\end{align}
where $C'$ can be made arbitrary large by large enough choice of $C_a$. Since $W_{[h]}^{(k)}\mid T^{(k)}$ is sub-exponential, Lemma \ref{lem:sub:exp:tail} gives, for all $T^{(k)}\in  \mathcal{C}^{(k)}$,
\begin{align*}
    \mathbb{P}(W_{[h]}^{(k)} \geq a_n \mid T^{(k)}) &= \mathbb{P}\left(W_{[h]}^{(k)} - \mathbb{E}[W_{[h]}^{(k)} \mid T^{(k)}] \geq a_n - \mathbb{E}[W_{[h]}^{(k)} \mid T^{(k)}] \Big| T^{(k)}\right) \\
    &\leq \exp\left(-\frac{(a_n - \mathbb{E}[W_{[h]}^{(k)}| T_h^{(k)}])^2}{\sum_{i \in \mathcal{T}_h^{(k)}} V(T_i^{(k)})}\right)\leq e^{-C' \log N},
\end{align*}
providing our statement.

\textbf{Proof of \eqref{cond1:lem10}.} 
By picking $C_a$ at least twice as large as the constant in the definition of $\mathcal{C}_{2h}^{(k)}$, we get $0<a_n/2<a_n - \mathbb{E}[W_{[h]}^{(k)} \mid T^{(k)}=t]$, for $t\in \mathcal{C}_{2h}^{(k)}$. For the upper bound note that
on the event $\mathcal{C}_{3h}^{(k)}\cap \mathcal{C}_1^{(k)}$,
\begin{align*}
    \sum_{i=1}^n V(T_i^{(k)})\mathbbm{1}_{T_i^{(k)} \in I_h} &\geq \sum_{i=1}^n \mathbb{E}\left[V(T_i^{(k)})\mathbbm{1}_{T_i^{(k)} \in I_h}\right] - C\eps^2 d^{-2r-1/2}\sqrt{n\log N } \\
    &\geq n c_1d^{-1-2r} - C\eps^2d^{-2r-1/2}\sqrt{n\log N } \geq (c_1/2)nd^{-1-2r}
\end{align*}
for $n$ large enough. Similarly we also get the upper bound $  \sum_{i \in \mathcal{T}_h^{(k)}} V(T_i^{(k)}) \leq C_2nd^{-1-2r}$ for some $C_2>0$. Therefore in view of the assumption $\max_{t} b(t)g_1(t) \leq Cd^{-r}$, 
\begin{align*}
 \frac{\sum_{i \in \mathcal{T}_h^{(k)}} V(T_i^{(k)})}{\max_{i \in \mathcal{T}_h^{(k)}} b(T_i^{(k)})g_1(T_i^{(k)})}&\gtrsim \frac{nd^{-1-2r}}{d^{-r}}\gtrsim \eps d^{-1/2-r} \sqrt{n\log N}\\
&\gtrsim a_n\geq a_n- \mathbb{E}[W_{[h]}^{(k)} \mid T^{(k)}],
\end{align*}
where in the last inequality we used that $\mathbb{E}[W_{[h]}^{(k)} \mid T^{(k)}]$ is a KL-divergence, hence nonnegative.

\textbf{Proof of \eqref{cond2:lem10}}. Using the bounds derived above, we have
\begin{align*}
     \frac{\big(a_n - \mathbb{E}[W_{[h]}^{(k)} \mid T^{(k)}]\big)^2}{\sum_{i \in \mathcal{T}_h^{(k)}} V(T_i^{(k)})} \geq \frac{a_n^2/4}{C_2nd^{-1-2r}} \geq C_a^2 \eps^2/(4C_2) \log{N}.
\end{align*}
\end{proof}

\subsection{Proof of Lemma \ref{lem:upper-bound-assumption-consequences}}\label{sec:proof-sufficient:upper}

Using iterated expectations we have
    \begin{align*}
        \mathbb{E}\left[\hat{f}_{n,jh}^{(k)}\right] = \mathbb{E}\left[\mathbb{E}[h(X_1^{(k)}) \mid T_1^{(k)}] \psi_{jh}(T_1^{(k)})]\right] = \int_0^1 f(t)\psi_{jh}(t)\,dt = f_{jh}.
    \end{align*}
Next note that by Jensen's inequality for $q\geq 1$, $\mathbb{E}[h(X)^2 \mid T] \leq (\mathbb{E}[|h(X)|^q \mid T])^{2/q}$, which in turn implies $\text{var}[h(X) \mid T] \leq \mathbb{E}[h(X)^2 \mid T]<\infty$. Take $C=\sup_{t \in [0,1]}\max\{f(t)^2, \text{var}[h(X) \mid t]\}$. By the law of total variance and $\mathbb{E}[\psi_{jh}(T)^2] =1$ we have
    \begin{align*}
        n\text{var}\left(\hat{f}_{n,jh}^{(k)}\right) = \mathbb{E}\left[\psi_{jh}(T_1^{(k)})^2\text{var}(h(X_1^{(k)}) \mid T_1^{(k)})\right] + \text{var}\left(f(T_1^{(k)})\psi_{jh}(T_1^{(k)})\right) \leq 2C,
    \end{align*}
which means that $\text{var}(\hat{f}_{n,jh}^{(k)}) \leq 2C/n$ across all $n,j$ and $h$, concluding the first part of Assumption \ref{ass:UB}.

Next consider the asymptotics of the tail bound. In view of Markov's inequality
\begin{align*}
    \mathbb{P}(\hat{f}_{n,jh}^{(k)} > \sqrt{N}) \leq \mathbb{P}\left(|\hat{f}_{n,jh}^{(k)}|^q > N^{q/2}\right) \leq \frac{\mathbb{E}[|\hat{f}_{n,jh}^{(k)}|^q]}{N^{q/2}}.
\end{align*}
Then by Minkowski's inequality we get
\begin{align*}
    \mathbb{E}[|\hat{f}_{n,jh}^{(k)}|^q] &\leq \mathbb{E}[|h(X_1^{(k)})\psi_{jh}(T_1^{(k)})|^q]\\
& = \mathbb{E}\left[\mathbb{E}[|h(X_1^{(k)})|^q \mid T_1^{(k)}] \psi_{jh}^2(T_1^{(k)}) \vert \psi_{jh}(T_1^{(k)})\vert^{q-2}\right]\\
& \leq C2^{(q-2)j/2}\mathbb{E}[\psi_{jh}^2(T_1^{(k)})] \leq C\Nbass^{\frac{(q-2)/2}{1 + 2r}}.
\end{align*} 
Putting this into the Markov inequality bound gives
\begin{align*}
     \mathbb{P}(\hat{f}_{n,jh}^{(k)} > \sqrt{N}) &\lesssim \frac{\Nbass^{\frac{(q-2)/2}{1 + 2r}}}{N^{q/2}} \lesssim N^{-\frac{qr + 1}{1 + 2r}}= o(m^{-1}N^{-\frac{3+4r}{1+2r}})
\end{align*}
since $m=o(N)$.

\section{Verification of examples}\label{sec:proof:examples}

Here we verify the conditions for a variety of models to demonstrate the broad applicability of our sufficient conditions. For $f\in B_{22}^r(L)$ we consider the sieve \eqref{def: sieve_disjoint} with $d=c2^{j_n}$. Note that $\|\fbb+C_0\|_\infty=O(1)$.

\subsection{Proof of Proposition \ref{prop:nonparametric-regression} (nonparametric regression)}\label{sec: proof:prop:regression}
This model can be written in the form \eqref{eq:model:reg}, hence it is sufficient to verify the assumptions given in Section \ref{sec:suff-conds}.\\
 
\textbf{Lower bound:} We verify that Assumption \ref{ass:expo_design} holds which in view of Lemma \ref{lem:sub_exp_cond_prob_bound} implies Assumptions \ref{ass:W-expect-control} and \ref{ass:tail} and as a consequence our proposition follows from Theorem \ref{thm:gen_lower_bound_thm}. For nonparametric regression we use the sieve $\mathcal{F}_0(d,C_0,\eps)$ given in \eqref{def: sieve_disjoint} with $C_0=0$, $d=c\Nbass^{1/(1+2r)}$ and $\eps>0$ small enough. Then
a term in the log-likelihood-ratio for $T_i^{(k)}\in I_h$
\begin{align*}
    W_{d,i}^{(k)} = \log\frac{p(X_i^{(k)}  \mid f_{j_nh} = \beta_h, T_i^{(k)} )}{p(X_i^{(k)}  \mid f_{j_nh} = -\beta_h, T_i^{(k)} )} &= 2X_i^{(k)} \fbb(T_i^{(k)} ),
\end{align*}
hence $h(X_i^{(k)}) = X_i^{(k)}, g_1(T_i^{(k)}) = 2\fbb(T_i^{(k)})$, and $g_2(T_i^{(k)}) = 0$. Conditional on $T_i^{(k)}$ the distribution of $X_i^{(k)}$ is Gaussian which implies sub-exponentiality with $\nu^2(T_i^{(k)}) = 1$ and $b(T_i^{(k)}) = 0$. Since $\mathbb{E}[X_i^{(k)} \mid T_i^{(k)}] = \fbb(T_i^{(k)})$ we have
\begin{align*}
    \mathbb{E}[W_{d,i}^{(k)} \mid T_i^{(k)}] &= 2\epsilon^2d^{-1-2r}\sum_{h=1}^d\psi_{jh}^2(T_i^{(k)}).
\end{align*}
By construction $|\psi_{j_nh}^2(T_i^{(k)})/d| \leq c^{-1}\Vert \psi \Vert_\infty^2$ and using that the basis $\psi_{j_nh}$, $h=1,...,d$ has disjoint support we get that
\begin{align*}
    \max_{T_i^{(k)}} \mathbb{E}[W_{d,i}^{(k)} \mid T_i^{(k)}] - \min_{T_i^{(k)}} \mathbb{E}[W_{d,i}^{(k)} \mid T_i^{(k)}] \leq 4\epsilon^2d^{-2r}c^{-1}\Vert \psi\Vert_\infty^2 =O(d^{-2r}).
\end{align*}
For $V(T_i^{(k)}) = \nu^2(T_i^{(k)})g_1^2(T_i^{(k)}) = 4\fbb(T_i^{(k)})^2$ we also have
\begin{align*}
    \max_{T_i^{(k)}} V(T_i^{(k)}) - \min_{T_i^{(k)}} V(T_i^{(k)}) \leq 8\epsilon^2d^{-2r}c^{-1}\Vert \psi \Vert_\infty^2. 
\end{align*}
As $X_i^{(k)}$ is sub-Gaussian, we trivially have $b(T_i^{(k)})g_1(T_i^{(k)}) = 0 =O(d^{-r})$. Thus we have verified all needed almost sure bounds. We finish the proof by showing that the bounds in expectation also hold, i.e.
\begin{align*}
    &\mathbb{E}[W_{d,i}^{(k)}]= \int_0^1 \mathbb{E}[W_{d,i}^{(k)} \mid T_i^{(k)}=t]\,dt = 2\epsilon^2 d^{-1-2r}\sum_{h=1}^d \int_0^1 \psi_{j_nh}^2(t)dt = 2\epsilon^2d^{-2r}, \\
    &\mathbb{E}[V(T_i^{(k)})\mathbbm{1}_{T_i^{(k)} \in I_h}] = \int_{I_h} 4\fbb(t)^2\,dt =  4\epsilon^2d^{-1-2r}.
\end{align*}

\noindent\textbf{Upper bound:} We verify Assumption \ref{ass:UB-suff-cond}, providing sufficient conditions for Theorem \ref{thm:gen_upper_bound_thm} and implying the stated upper bound. Note that $h(X) = X$ and $\mathbb{E}[X \mid T] = f(T)$. As the density of $X \mid T$ is continuous in $T$ and the normal distribution has all moments, we see that $\mathbb{E}[|h(X)|^q \mid T]$ is bounded for $T \in [0,1]$.

\subsection{Proof of Proposition \ref{prop:density-estimation} (density estimation)}\label{sec: proof:prop:density}
This model falls outside of the regression framework \eqref{eq:model:reg}. Therefore we verify Assumptions  \ref{ass:W-expect-control}, \ref{ass:structural} and \ref{ass:tail} for the lower bound and Assumption \ref{ass:UB} for the upper bound.\\

\textbf{Lower bound:}  We consider the testing sieve $\mathcal{F}_0(d,C_0,\eps)$ given in \eqref{def: sieve_disjoint} with $C_0=1$, $d=c2^{j_n}=cN_{bass}^{1/(1+2r)}$ and $1 \geq \eps>0$ small enough that $1+f_{\bbeta}(x)>0$ for all $\bbeta\in\{-1,1\}^d$. First note that by considering the partition $[0,1]=\cup_{h=1}^d I_h$ we have that $\mathbb{P}_{\bbeta}(I_h)=1/d$ for all $\bbeta\in \{-1,1\}^d$ due to symmetry and the disjointedness of the support of the selected wavelets. Furthermore, the distribution of $X^{(k)}_i$ given that $X^{(k)}_i\in I_h$ only depends on $\beta_h$ and not on $\beta_{h'}$, $h'\neq h$. Hence Assumption \ref{ass:structural} holds.

We now consider Assumption \ref{ass:W-expect-control}. The conditional density of $p(x \mid x \in I_h, \beta_h)$ is equal to $d(1+\fbb)\mathbbm{1}_{x \in I_h}$. Using the fact that the intervals $\{I_h\}_{h=1}^d$ form a partition, we see
\begin{align*}
    \mathbb{E}[W_{d,i}^{(k)}] = \sum_{h=1}^d \frac{1}{d}\int_{I_h} d(1+\fbb)W_{d,i}^{(k)}\,dx = \frac{1}{d}\sum_{h=1}^d KL\left(p(x \mid x \in I_h, \beta_h),p(x \mid x \in I_h, \beta_h')\right).
\end{align*}
The divergences on the right are equal to zero when $\beta_h = \beta_h'$ and strictly positive otherwise. This implies that for any $\bbeta$ the divergence between $1 + \fbb$ and $1 + \fbbp$ is maximized when $\beta_h \neq \beta_h'$ for all $h$, i.e. $\beta' = - \beta$. We can therefore focus on bounding the divergence between $1 + \fbb$ and $1 + \fnb = 1 - \fbb$. By construction the testing sieve satisfies $\vert\fbb\vert \leq \epsilon d^{-r}\Vert \psi \Vert_\infty/\sqrt{c} \leq \Vert \psi \Vert_\infty/\sqrt{c}$ when $\epsilon \leq 1$ and $d^{-r} \leq 1$. Using Lemma \ref{lem:log-ratio-bound} we have
\begin{align*}
2\fbb - 4f_{\bbeta}^2 \leq W_{d,i}^{(k)} &= \log\frac{1 + \fbb}{1 - \fbb} \leq 2\fbb + 4f_{\bbeta}^2. 
\end{align*}
We now calculate the expectation of $2\fbb + 4\fbb^2$, which will give the upper bound on $\mathbb{E}[W_{d,i}^{(k)}]$. We see that
\begin{align*}
    \mathbb{E}_{\bbeta} \left[\psi_{j_nh}(X_i^{(k)})\right] &= \int_0^1 \psi_{j_nh}(x)(1 +\fbb)\,dx = 0 + \epsilon d^{-r - 1/2}\beta_h \leq \epsilon d^{-r - 1/2}, \\
    \mathbb{E}_{\bbeta} \left[\psi_{j_nh}(X_i^{(k)})^2\right] &= \int_0^1 \psi_{j_nh}^2(x)(1 +\fbb)\,dx \leq 1 + \epsilon d^{-r}\frac{\Vert \psi \Vert_\infty}{\sqrt{c}}. 
\end{align*}
Using these results we calculate
\begin{align*}
    \mathbb{E}_{\bbeta}\left[W_{d,i}^{(k)}\right] &\leq 2\mathbb{E}[\fbb] + 4\mathbb{E}[f_{\bbeta}^2] \leq 2d(\epsilon d^{-r-\frac{1}{2}})(\epsilon d^{-r - \frac{1}{2}}) + 4d(\epsilon^2 d^{-2r - 1})(1 + \epsilon d^{-r}\frac{\Vert \psi \Vert_\infty}{\sqrt{c}}) \\
    &\leq  2\epsilon^2 d^{-2r} + 4\epsilon^2d^{-2r} + 4\epsilon^2d^{-3r}\frac{\Vert \psi \Vert_\infty}{\sqrt{c}} \leq C_r\epsilon^2 d^{-2r},
\end{align*}
for a constant $C_r$ that only depends on $r$. 

Now we consider Assumption \ref{ass:tail}. The event $E_h^{(k)}(a_n)$ only depends on $\beta_h$ and $\beta_h'$. If $\beta_h = \beta_h'$, then the probability of the event is zero because the log-likelihood ratio is zero. So consider $\beta_h \neq \beta_h'$. This means the bounds derived above a fortiori apply for $X_i^{(k)} \in I_h = \chi_h$. 


Let us introduce the latent variable $U_i^{(k)}$  such that  $U_i^{(k)}=h$ if $X_i^{(k)}\in I_h$ and let $U^{(k)}= (U_i^{(k)})_{i=1,...,n}$. Define the event 
\begin{align*}
    D_h^{(k)} = \left[ \frac{n}{2d} \leq \sum_{i=1}^n \mathbbm{1}_{U_i^{(k)} = h} \leq \frac{2n}{d}\right],
\end{align*}
which requires the number of observations in $I_h$ to be close to the expected value of $n/d$. In view of Chernoff's bound (see for instance Lemma 5.2 of \cite{szabo_adaptive_2019} for its application to the present setting), we know that $\mathbb{P}(D_h^{(k)}) = 1 + o(m^{-2}d^{-2})$. Using the law of total probability we have
\begin{align*}
    \mathbb{P}(E_h^{(k)}(a_n))  \leq \mathbb{P}(E_h^{(k)}(a_n)\mid D_h^{(k)}) + o(m^{-2}d^{-2}),
\end{align*}
so it suffices to bound the conditional probability to be $o(m^{-2}d^{-2})$.  We now want to apply Lemma \ref{lem:gen-prob-lemma} to bound the probability. The sequence $X_i^{(k)}$ is iid for $i = 1,\ldots, n$ and we take the function $h(X_i^{(k)}) = W_d^{(k)}(X_i^{(k)}, \beta,\beta_h)$. The sum $S_n = \sum_{i=1}^n h(X_i^{(k)})\mathbbm{1}_{X_i^{(k)} \in I_h}$ is exactly the quantity for which we seek to establish a tail bound. Let $\chi_n = I_h$ be the set of interest. The length of the support of $W_{d,i}^{(k)} = h(X_i^{(k)})$ conditional on $X_i^{(k)} \in I_h$ is bounded by $C_r'\epsilon^2 d^{-2r}$ for a constant $C_r'$ that only depends on $r$. From $\mathbb{P}(X_i^{(k)} \in I_h) = 1/d$ and the above calculation we get
\begin{align*}
   \mathbb{E}[W_{d,i}^{(k)} \mid X_i^{(k)} \in I_h] = \int_{I_h} W_{d}^{(k)}(x, \beta_h, \beta_h')d(1 + \psi_{j_nh}(x))\,dx \leq C_r\epsilon^2d^{-2r}.
\end{align*}
If we let $\delta_n = d^{-r}$, then Lemma \ref{lem:gen-prob-lemma} holds with $r_n = n/d$, $I = \{i: U_i^{(k)} = h\}$, and $\alpha = 2$. The lemma gives a bound of $\exp(-C' d^{2r}\log{N}) = o(m^{-2}d^{-2})$ for each of the finitely many combinations of $\beta_h$ and $\beta_h'$. This establishes that Assumption \ref{ass:tail} holds.\\

\textbf{Upper bound:} We verify Assumption \ref{ass:UB} directly. Consider the local estimator $\hat{f}_{n,jh}^{(k)} = \frac{1}{n}\sum_{i=1}^n \psi_{jh}(X_i^{(k)})$. We see that $\mathbb{E}[\hat{f}_{n,jh}^{(k)}] = \int \psi_{jh}(x)f(x) = f_{jh}$. As the target density $f$ is continuous on a compact interval, we can pick a $C$ such that $f(x) \leq C$ for all $x \in [0,1]$. The variance of $\hat{f}_{n,jh}^{(k)}$ is bounded by the second moment, which satisfies $\text{var}(\hat{f}_{n,jh}^{(k)}) \leq \frac{1}{n}\int \psi_{jh}^2f \leq C/n$. Now we can establish the tail bounds using Hoeffding's inequality. As the wavelet $\psi_{jh}$ is bounded by $|\psi_{jh}| \leq CN^{1/(1+2r)}$, we see that the estimator satisfies
\begin{align*}
    \mathbb{P}\left(\hat{f}_{n,jh}^{(k)} > \sqrt{N}\right) \leq \exp\left(-\frac{2n^2\left(\sqrt{N} - f_{jh}\right)^2}{2nCN^{1/(1+2r)}}\right) \leq e^{-Cn} = o(m^{-1}N^{(3+2r)/(1+2r)})
\end{align*}
as long as $m$ is polynomial in $n$. Therefore all parts conditions of Assumption \ref{ass:UB} are satisfied and Theorem \ref{thm:gen_upper_bound_thm} holds. 


\subsection{Proof of Proposition \ref{prop:binary-regression} (binary regression)} \label{sec: proof:prop:binary}
This model can be written in the form \eqref{eq:model:reg}, hence it is sufficient to verify the assumptions given in Section \ref{sec:suff-conds}.\\

\noindent{\textbf{Lower bound:}} Similarly to the proof of Proposition \ref{prop:nonparametric-regression} it is sufficient to show that Assumption \ref{ass:expo_design} holds.

We consider the sieve $\mathcal{F}_0(d,C_0,\eps)$ given in \eqref{def: sieve_disjoint} with $C_0=1/2$, $d=c2^{j_n}=c\Nbass^{1/(1+2r)}$ and $\eps>0$ small enough to ensure that $0 \leq  \tilde{\fbb}=\fbb+1/2 \leq 1$. Note that $1 - \tilde{\fbb} = \tfnb$, then a term in the log-likelihood-ratio is
\begin{align*}
    W_{d,i}^{(k)} &= \log \frac{\tilde{\fbb}(T_i^{(k)})^{X_i^{(k)}}\big(1 - \tilde{\fbb}(T_i^{(k)})\big)^{1-X_i^{(k)}}}{ \tfnb(T_i^{(k)})^{X_i^{(k)}}\big(1- \tfnb(T_i^{(k)})\big)^{1 - X_i^{(k)}}}\\
& = (2X_i^{(k)} - 1) \log\frac{\tilde{\fbb}}{ \tfnb}(T_i^{(k)})
\end{align*}
hence $h(x) = 2x-1, g_1(t) = \log(\tilde{\fbb}/ \tfnb)(t)$, and $g_2(t) = 0$.

The expectation conditional on $T_i^{(k)}$ is
\begin{align*}
    \mathbb{E}[W_{d,i}^{(k)} \mid T_i^{(k)}] = 2\fbb(T_i^{(k)})\log\frac{\tilde{\fbb}(T_i^{(k)})}{ \tfnb(T_i^{(k)})} = 2\fbb(T_i^{(k)})\log\frac{1 + 2\fbb(T_i^{(k)})}{1 - 2\fbb(T_i^{(k)})}.
\end{align*}
In view of Lemma \ref{lem:log-ratio-bound}, when $\epsilon$ is small enough to ensure $\fbb \leq 1/4$, we have 
\begin{align*}
   0 \leq \mathbb{E}[W_{d,i}^{(k)} \mid T_i^{(k)}] \leq 16f_{\bbeta}^2 \lesssim \eps^2 d^{-2r}.
\end{align*}
This establishes that the length of the support is $O(d^{-2r})$. By Hoeffding's lemma, any bounded random variable with support $[a,b]$ is sub-Gaussian with parameter $(b-a)^2/4$, so we know $(h(X_i^{(k)})|T_i^{(k)})= (2X_i^{(k)} -1|T_i^{(k)}) \sim SE(1,0)$. Furthermore, again in view of Lemma \ref{lem:log-ratio-bound} and by noting that $\nu^2(T_i^{(k)})=1$, we have
\begin{align*}
   \Big| \log\frac{1+2\fbb(T_i^{(k)})}{1-2\fbb(T_i^{(k)})} \Big|=O(d^{-r}) \implies 0 \leq \nu^2(T_i^{(k)})g_1^2(T_i^{(k)}) \leq Cd^{-2r}.
\end{align*}
We also have $g_1(T_i^{(k)})b(T_i^{(k)}) = 0=O(d^{-r}) $, as needed. 

Next note that  $\mathbb{E}[W_{d,i}^{(k)} \mid T_i^{(k)}] \leq cd^{-2r}$ implies $\mathbb{E}[W_{d,i}^{(k)}]=O(d^{-2r})$. Finally we establish a lower bound on $\mathbb{E}[\nu^2(T_i^{(k)})g_1^2(T_i^{(k)})\mathbbm{1}_{T_i^{(k)} \in I_h}]$. When $\epsilon$ is small enough that $|2\fbb| \leq (\sqrt{2} - 1)/(\sqrt{2} + 1)$, we can apply Lemma \ref{lem:log-ratio-squared} to get 
\begin{align*}
    g_1(T_i^{(k)})^2 = \left(\log\frac{1+2\fbb(T_i^{(k)})}{1-2\fbb(T_i^{(k)})}\right)^2 \geq 2\fbb(T_i^{(k)})^2. 
\end{align*}
Taking the expectation gives $\mathbb{E}[g_1^2(T_i^{(k)})\mathbbm{1}_{T_i^{(k)} \in I_h}] \geq 2 \int_{I_h} \fbb^2(t) dt\gtrsim d^{-1-2r}$, finishing the proof of the statement. \\

\noindent\textbf{Upper bound:} As with the nonparametric regression case it is sufficient to prove that the conditions of Assumption \ref{ass:UB-suff-cond} hold. We have $X \mid T \sim \text{Bern}(f(T))$, hence we take $h(X) = X$ from which it is immediate that $\mathbb{E}[X \mid T] = f(T)$. The function $h(X)$ is bounded, so the moment condition is also satisfied.  


\subsection{Proof of Proposition \ref{prop:poisson-regression} (Poisson regression)}\label{sec: proof:prop:poisson}
This model can be also written in the form \eqref{eq:model:reg}, hence it is sufficient to verify the assumptions given in Section \ref{sec:suff-conds}.\\

\textbf{Lower bound:} Similarly to before it is sufficient to show that Assumption \ref{ass:expo_design} holds. We consider the sieve $\mathcal{F}_0(d,C_0,\eps)$ given in \eqref{def: sieve_disjoint} with $C_0=1$, $d=c2^{j_n}=c\Nbass^{1/(1+2r)}$ and $\eps>0$ small enough to ensure that $0 < \fbb $.

First note that the log-likelihood-ratio is
\begin{align*}
    W_{d,i}^{(k)} &= X_i^{(k)}\log\frac{1 + \fbb(T_i^{(k)})}{1 - \fbb(T_i^{(k)})} - 2\fbb(T_i^{(k)})\\
& = \left(X_i^{(k)} - (1 + \fbb(T_i^{(k)}))\right)\log\frac{1 + \fbb(T_i^{(k)})}{1-\fbb(T_i^{(k)})}\\
&\qquad\qquad - \left(2\fbb(T_i^{(k)}) -(1+\fbb(T_i^{(k)}))\log\frac{1+\fbb(T_i^{(k)})}{1-\fbb(T_i^{(k)})}\right), 
\end{align*}
{hence $g_1(t)=\log\frac{1 + \fbb}{1-\fbb}(t)$, $g_2(t)= 2\fbb(t) -(1+\fbb(t))\log\frac{1+\fbb(t)}{1-\fbb(t)}$ and $h(x) = x - (1 + \fbb(t))$}. 

First note that in view of Lemma \ref{lem:Poisson-sub-expo}, $h(X_i^{(k)})|T_i^{(k)} = X_i^{(k)} - \big(1 + \fbb(T_i^{(k)})\big)|T_i^{(k)}$ is sub-exponential with parameters $\nu^2(T_i^{(k)}) = 2(1 + \fbb(T_i^{(k)}))$ and $b(T_i^{(k)})=1$. The conditional expectation is
\begin{align*}
    \mathbb{E}[W_{d,i}^{(k)} \mid T_i^{(k)}] =-g_2(T_i^{(k)})=-2\fbb (T_i^{(k)})+\big(1+\fbb(T_i^{(k)})\big)\log\frac{1+\fbb(T_i^{(k)})}{1-\fbb(T_i^{(k)})}.
\end{align*}
Then by Lemma \ref{lem:log-ratio-bound}
\begin{align*}
    &2\fbb -(1+\fbb)\log\frac{1+\fbb}{1-\fbb} \geq 2\fbb - (1+\fbb)(2\fbb + 4\fbb^2) = -6\fbb^2 - 4\fbb^3 \gtrsim -d^{-2r},\\
    &2\fbb -(1+\fbb)\log\frac{1+\fbb}{1-\fbb} \leq 2\fbb - (1 + \fbb)(2\fbb - 4\fbb^2) = 2\fbb^2 + 4\fbb^3 \lesssim d^{-2r},
\end{align*}
resulting in a bounded support of sufficiently short length
\begin{align*}
    \max_{t}\mathbb{E}[W_{d,i}^{(k)} \mid T_i^{(k)}=t] - \min_{t} \mathbb{E}[W_{d,i}^{(k)} \mid T_i^{(k)}=t] =O( d^{-2r}).
\end{align*}

We next consider the term $\nu^2(t)g_1^2(t)\geq 0$. Taking square roots in Lemma \ref{lem:log-ratio-squared} we have
\begin{align*}
 \Big|\log\frac{1+\fbb(t)}{1-\fbb(t)}\Big| \leq Cd^{-r},
\end{align*}
which in turn implies 
\begin{align*}
  0 \leq \nu^2(t)g_1^2(t) \leq 2(1 + Cd^{-r})(Cd^{-2r}) \lesssim Cd^{-2r}.
\end{align*}
It also follows that $g_1(t)b(t) = g_1(t) \leq Cd^{-r}$.

Next we consider $\mathbb{E}[\nu^2(T_i^{(k)})g_1^2(T_i^{(k)})\mathbbm{1}_{T_i^{(k)} \in I_h}]$. Then in view of Lemma \ref{lem:log-ratio-squared} we have (for $n$ large enough or $\epsilon$ small enough to ensure $\Vert \fbb \Vert_\infty \leq (\sqrt{2} - 1)/(\sqrt{2} + 1)$)
\begin{align*}
    \nu^2g_1^2 = 2(1 + \fbb)\left[\log\frac{1+\fbb}{1-\fbb}\right]^2 \geq (1 + \fbb)\fbb^2 \geq \frac{2\fbb^2}{\sqrt{2} + 1}.
\end{align*}
Integrating against the uniform distribution gives
\begin{align*}
    \mathbb{E}[\nu^2(T_i^{(k)})g_1^2(T_i^{(k)})\mathbbm{1}_{T_i^{(k)} \in I_h}]\geq \frac{2}{\sqrt{2} + 1}\int_{I_h} \fbb^2(t) \,dt \gtrsim d^{-1-2r}.
\end{align*}

 Finally we show $\mathbb{E}[W_{d,i}^{(k)}] \lesssim d^{-2r}$. Using Lemma \ref{lem:log-ratio-bound} we have
\begin{align*}
    \mathbb{E}[W_{d,i}^{(k)} \mid T_i^{(k)}] &= \big(1 + \fbb(T_i^{(k)})\big)\log\frac{1+\fbb(T_i^{(k)})}{1-\fbb(T_i^{(k)})} - 2\fbb(T_i^{(k)})\\
& \leq \big(1+\fbb(T_i^{(k)})\big)(2\fbb(T_i^{(k)}) + 4\fbb(T_i^{(k)})^2) - 2\fbb(T_i^{(k)})\\
& = 6\fbb(T_i^{(k)})^2 + 4\fbb(T_i^{(k)})^3,
\end{align*}
which in turn implies that
\begin{align*}
    \mathbb{E}[W_{d,i}^{(k)}] = \mathbb{E}\left[\mathbb{E}[W_{d,i}^{(k)} \mid T_i^{(k)}]\right] \leq \int_0^1 6\fbb(t)^2 + 4\fbb(t)^3\,dt \lesssim d^{-2r}.
\end{align*}

\textbf{Upper bound:} To establish Theorem \ref{thm:gen_upper_bound_thm} we use the sufficient conditions in Assumption \ref{ass:UB-suff-cond}. Note that for $h(X) = X$ we have $\mathbb{E}[X \mid T] = f(T)$. In addition, the distribution of $h(X) \mid T$ is Poisson, so we know that all moments are finite. Combined with the continuity of the model in $T$, this is sufficient to establish the second point of the assumption.

\subsection{Proof of Proposition \ref{prop:heteroskedastic-regression} (heteroskedastic regression)}\label{sec: proof:prop:heteroskedastic}
This model can  also be written in the form \eqref{eq:model:reg}, hence it is sufficient to verify the assumptions given in Section \ref{sec:suff-conds}.\\

\textbf{Lower bound:} Similarly as before it is sufficient to show that Assumption \ref{ass:expo_design} holds. We consider the sieve $\mathcal{F}_0(d,C_0,\eps)$ given in \eqref{def: sieve_disjoint} with $C_0=1$, $d=c2^{j_n}=c\Nbass^{1/(1+2r)}$ and $\epsilon>0$ is chosen to ensure that $0 < 1+\fbb$. Since $X_i^{(k)} \mid T_i^{(k)} \sim N(0, 1 + \fbb(T_i^{(k)}))$, the corresponding log-likelihood-ratio is
\begin{align*}
    W_{d,i}^{(k)} = \frac{(X_i^{(k)})^2}{2}\left(\frac{1}{1- \fbb(T_i^{(k)})} - \frac{1}{1 + \fbb(T_i^{(k)})}\right) - \frac{1}{2}\log\frac{1 + \fbb(T_i^{(k)})}{1 - \fbb(T_i^{(k)})},
\end{align*}
hence $ h(X_i^{(k)}) = (X_i^{(k)})^2/2$, $g_1(T_i^{(k)})=\left(\frac{1}{1- \fbb(T_i^{(k)})} - \frac{1}{1 + \fbb(T_i^{(k)})}\right)$, $g_2(T_i^{(k)})=\frac{1}{2}\log\frac{1 + \fbb(T_i^{(k)})}{1 - \fbb(T_i^{(k)})}$.

Note that by Lemma \ref{lem:prop-of-sub-expo-rvs},
$$h(X_i^{(k)})|T_i^{(k)} \stackrel{d}{=} (1 + \fbb(T_i^{(k)})) Z^2/2 \sim SE\left((1 + \fbb(T_i^{(k)}))^2, 2(1+ \fbb(T_i^{(k)}))\right),$$ where $Z \sim N(0,1)$,  hence $Z^2 \sim SE(4,4)$. Furthermore, the conditional expectation is
\begin{align*}
    \mathbb{E}[W_{d,i}^{(k)} \mid T_i^{(k)}] =\frac{1}{2}\Big[\frac{1 + \fbb(T_i^{(k)})}{1 - \fbb(T_i^{(k)})} - 1\Big] - \frac{1}{2}\log \frac{1+ \fbb(T_i^{(k)})}{1 - \fbb(T_i^{(k)})}.
\end{align*}
Using Lemma \ref{lem:ratiobound} and Lemma \ref{lem:log-ratio-bound} we have
\begin{align*}
    2\fbb + \fbb^2 \leq \frac{1+\fbb}{1-\fbb} - 1 \leq 2\fbb + 4\fbb^2, \\
    2\fbb - 4\fbb^2 \leq \log\frac{1+\fbb}{1-\fbb} \leq 2\fbb + 4\fbb^2.
\end{align*}
Combining these bounds gives
\begin{align*}
    -\frac{3}{2}\fbb(T_i^{(k)})^2 \leq \mathbb{E}[W_{d,i}^{(k)} \mid T_i^{(k)}] \leq 4\fbb(T_i)^2,
\end{align*}
which in turn implies
\begin{align*}
\max_{t} \mathbb{E}[W_{d,i}^{(k)} \mid T_i^{(k)}=t] - \min_{t} \mathbb{E}[W_{d,i}^{(k)} \mid T_i^{(k)}=t] \leq O(d^{-2r} ).
\end{align*}
For the quantity $\nu^2g_1^2= (\frac{1+\fbb}{1-\fbb} - 1)^2$ we can use the above bound to get
\begin{align}
4\fbb^2 - 16|\fbb|^3 \leq \nu^2g_1^2 \leq 4\fbb^2 + 16|\fbb|^3 + 16\fbb^4. \label{eq:nu-g1-hetero-bound}
\end{align}
Since the term $4\fbb^2$ dominates both sides the support is of order $O(d^{-2r})$. Furthermore,
\begin{align*}
    g_1(T_i^{(k)})b(T_i^{(k)}) = 2\left(\frac{1+\fbb(T_i^{(k)})}{1-\fbb(T_i^{(k)})} - 1\right) \leq  Cd^{-r}.
\end{align*}
Finally, note that the above established $\mathbb{E}[W_{d,i}^{(k)} \mid T_i^{(k)}] \leq 4\fbb^2$ implies $\mathbb{E}[W_{d,i}^{(k)}]=O(d^{-2r})$ and by multiplying \eqref{eq:nu-g1-hetero-bound} by $\mathbbm{1}_{T_i^{(k)} \in I_h}$ and taking expectations we have $c_1 d^{-1-2r}\leq \mathbb{E}[\nu^2g_1^2\mathbbm{1}_{T_i^{(k)} \in I_h}]$ concluding the verification of Assumption \ref{ass:expo_design}.\\ 

\textbf{Upper bound:} We establish that Assumption \ref{ass:UB-suff-cond} holds for this model. Since $X \sim N(0, f(T))$ we have for $h(X)=X^2$ that $\mathbb{E}[h(X) \mid T] = f(T)$. Note that unlike in the previous examples here a different function than $h(X)=X$ was taken. We note that $X^2 \mid T \sim f(T)\chi_1^2$, which has all moments since the MGF exists. The second condition of Assumption \ref{ass:UB-suff-cond} will therefore hold if $f(T)$ is continuous.

\section{Auxiliary lemmas}\label{sec:auxiliary}
In this section we collect technical lemmas used in the proof of our main results.

\subsection{Sub-Gaussian and sub-exponential random variables}\label{sec:subexp}
In this paper we make use of sub-Gaussian and sub-exponential random variables for their tail properties. We briefly review the features of these classes of random variables that are most relevant to our work. 

\begin{defn}\label{def: subexp}
A random variable $X$ is \emph{sub-Gaussian} if there exists a $\sigma^{2} > 0$ such that
\begin{align*}
\mathbb{E}\left[ e^{\lambda(X - \mathbb{E}[X])}\right] \leq e^{\sigma^{2}\lambda^{2}/2},\quad \forall\lambda\in\mathbb{R}
\end{align*}
in which case we write $X \sim \text{SG}(\sigma^{2})$. We say that $X$ is \emph{sub-exponential} if there exist $\nu^{2} > 0$ and $b \geq 0$ such that
\begin{align*}
\mathbb{E}\left[ e^{\lambda(X - \mathbb{E}[X])}\right] \leq e^{\nu^{2}\lambda^{2}/2} \quad \text{for } |\lambda| \leq \frac{1}{b},
\end{align*}
in which case we write $X \sim\text{SE}(\nu^{2},b)$. 
\end{defn}

Note that these random variables can posses non-zero means, only the above tail bound is required on the centered random variables. We allow $b = 0$ in the definition of sub-exponential random variables (taking $1/0 = \infty$, implying that the bound holds for all $\lambda$) so that $\text{SE}(\nu^{2}, 0) = \text{SG}(\nu^{2})$. Recall that bounded random variables $X \in [a,b]$ are sub-Gaussian with $\sigma^{2} = (b-a)^{2}/4$. 

Next we recall the Hoeffding's tail bound for sub-exponential random variables, see Proposition 2.9 of \cite{Wainwright} for a proof.
\begin{lem}[Sub-exponential tail bound]\label{lem:sub:exp:tail}
Suppose that $X \sim \text{SE}(\nu^{2}, b)$. Then
\begin{align*}
\mathbb{P}\left(X \geq \mathbb{E}[X]+ t\right) \leq \begin{cases} e^{-\frac{t^{2}}{2\nu^{2}}}, & 0 \leq t \leq \nu^{2}/b, \\
e^{-\frac{t}{2b}}, & t > \nu^{2}/b. \end{cases}
\end{align*}
\end{lem}
Note that a similar bound for $\mathbb{P}(|X - \mathbb{E}[X]| \geq t)$ holds by including a multiplicative factor of two on the right hand side. In the case when $b=0$, we see that the bound for sub-Gaussian random variables is $\exp(-t^{2}/(2\nu^{2}))$.  
\begin{lem}[Properties of sub-exponential random variables] \label{lem:prop-of-sub-expo-rvs}
If $X_i \stackrel{iid}{\sim} \text{SE}(\tau_i^2, b_i)$ then
\begin{align*}
    aX_i &\sim \text{SE}(a^2\tau_i^2, ab_i), \\
    \sum_{i=1}^n X_i &\sim \text{SE}\Big(\sum_{i=1}^n \tau_i^2, \max_i b_i\Big).
\end{align*}
\end{lem}
\begin{proof}
The second result is proved on page 29 of \cite{Wainwright}. The first result is straight-forward but not widely available. We have for the moment generating function of the centered random variable that
\begin{align*}
    M_{aX_i}(\lambda) = M_{X_i}(a\lambda) \leq \exp\Big(\frac{\lambda^2a^2\tau_i^2}{2}\Big), \qquad \text{for}\, |a\lambda| \leq 1/b_i,
\end{align*}
so the MGF is bounded for $|\lambda| \leq (ab_i)^{-1}$. 
\end{proof}

\begin{lem}[Poisson is sub-exponential] \label{lem:Poisson-sub-expo}
Suppose $X \sim \text{Pois}(\lambda)$, then $(X - \lambda) \sim SE(2\lambda, 1)$.
\end{lem}
\begin{proof}
Recall that $e^x \leq 1 + x + x^2$ for $|x| \leq 1$. The moment generating function of $X - \lambda$ is
\begin{align*}
    \mathbb{E}\left[e^{t(X - \lambda)}\right] &= e^{-t\lambda}\mathbb{E}\left[ e^{tX}\right]  = e^{-t\lambda}e^{\lambda(e^t - 1)}  \leq e^{-t\lambda}e^{\lambda t + \lambda t^2}  = e^{t^2\tau^2/2}
\end{align*}
for $|t| \leq 1$ and $\tau^2 = 2\lambda$.
\end{proof}

\subsection{Additional technical lemmas}
\begin{lem}[General sum bound] \label{lem:gen-prob-lemma}
{Suppose $X_1,\ldots, X_n$ are iid random variables. Let $\chi_n$ be a subset of the support of these random variables and define the corresponding random index set $I_n = \{i: X_i \in \chi_n\}$. Assume that for a measurable function $h$ there exists a sequence $\delta_n=o(1)$, such that
\begin{align*}
    \sup_{x \in \chi_n} h(x) - \inf_{x \in \chi_n} h(x) \leq c_1 \delta_n^\alpha \quad \text{and} \quad \mathbb{E}[h(X_i) \mid X_i \in \chi_n] \leq c_2\delta_n^2
\end{align*}
for some $\alpha \geq 1$ and positive constants $c_1$ and $c_2$. Define 
$$S_n =\sum_{i=1}^n h(X_i)\mathbbm{1}_{X_i \in \chi_n}$$
 to be the sum over $I_n$. Then the following hold:
\begin{enumerate}
    \item  For any realization of $I_n$ such that $|I_n| \leq c_3r_n$, with $\delta_n^2r_n < \log{N}$, we have $$\mathbb{E}[S_n \mid I_n] < c_2c_3\delta_n\sqrt{r_n\log{N}}.$$ 
    \item For $0\leq c_4<c_3$, $C>c_2c_3$ and $\delta_n^2r_n < \log{N}$ we have
    \begin{align}
        \mathbb{P}(S_n > C\delta_n\sqrt{r_n\log{N}} \mid  c_4r_n\leq |I_n| \leq c_3r_n) \leq e^{-C'\delta_n^{2 (1-\alpha)}\log{N}}, \label{eq:gen_sum_bound}
    \end{align}
    where $C'>0$ can be made arbitrarily large by taking $C$ large enough. 
\end{enumerate}}
\end{lem}
\begin{proof}
Using iterated expectations by conditioning on the realization of the set $I_n$, we get
\begin{align*}
    \mathbb{E}[S_n \mid I_n] = |I_n|\,\mathbb{E}[h(X_1) \mid X_1 \in \chi_n] \leq c_2c_3\delta_n^2r_n < c_2c_3\delta_n\sqrt{r_n\log{N}}.
\end{align*}
Next note that $S_n|I_n$ can be written as a sum of $|I_n|$ iid terms, whose support length is bounded by $c_1\delta_n^\alpha$. Then by Hoeffding's inequality
\begin{align*}
      \mathbb{P}(S_n > C\delta_n\sqrt{r_n\log{N}} \mid I_n) &=   \mathbb{P}\left(S_n - \mathbb{E}[S_n \mid I_n] > C\delta_n\sqrt{r_n\log{N}} - \mathbb{E}[S_n \mid I_n]\mid I_n \right)\\
&\leq \mathbb{P}\left(\left\vert S_n - \mathbb{E}[S_n \mid I_n]\right\vert > (C- c_2c_3)\delta_n\sqrt{r_n\log{N}}\right)\\
&\leq \exp\left(-\frac{(C - c_2c_3)^2\delta_n^2r_n\log{N}}{2|I_n|c_1^2\delta_n^{2\alpha}}\right) \leq \exp\left(-C'\delta_n^{2(1 - \alpha)}\log{N}\right),
\end{align*}
with $C' = (C - c_2c_3)^2/(2c_3c_1^2)$, which can be made arbitrarily large by increasing $C$. Equation \eqref{eq:gen_sum_bound} then follows from the law of total probability by averaging over the distribution of $I_n \mid  (c_4 r_n\leq |I_n| \leq c_3r_n)$. 
\end{proof}

We now provide a number of simple inequalities that allow us to relate the log-likelihood ratio in our examples to polynomials in the testing sieve functions $\fbb$. These results \emph{do not} have optimized constants or the largest possible range of validity in $x$.

\begin{lem} \label{lem:ratiobound}
For any $-1 \leq x \leq 1/2$ the following inequalities hold:
\begin{align*}
    1 + 2x + x^2 \leq \frac{1+x}{1-x} \leq 1 + 2x + 4x^2,\\
 x^2 \leq \left(\frac{1+x}{1-x} - 1\right)^2 \leq 16x^2.
\end{align*}
\end{lem}
\begin{proof}
By elementary algebra one can restate the first line as
\begin{align*}
0 \geq -x^2(1 + x)\qquad\text{and} \qquad  0 \leq 2x^2(1 - 2x),
\end{align*}
and the second line as 
\begin{align*}
(x-3)(x+1) \leq 0\qquad\text{and}\qquad (x-1/2)(x-3/2) \geq 0,
\end{align*}
 which all hold for $x\in[-1,1/2]$.
\end{proof}

\begin{lem} \label{lem:logbound}
For $|x| \leq 1/2$,
\begin{align*}
    -x - x^2 \leq \log(1-x).
\end{align*}
\end{lem}
\begin{proof}
Define the function
\begin{align*}
    h(x) = \log(1-x) + x + x^2.
\end{align*}
It is sufficient to show that $h(x) \geq 0$ for $|x| \leq 1/2$. The derivative of $h$ is
\begin{align*}
    h'(x) &= x(1 - 2x)(1 -x)^{-1}.
\end{align*}
Since $h'(x)$ is negative for $-1/2 \leq x < 0$ and non-negative for $0 \leq x \leq 1/2$, we see that the minimum is achieved at $x = 0$. As $h(0) = 0$, we have $h(x) \geq 0$, which proves the inequality. 
\end{proof}
We now combine the previous two lemmas to create a bound on $\log((1+\fbb)/(1-\fbb))$. 
\begin{lem}[Log ratio bound]\label{lem:log-ratio-bound}
For $|x| \leq 1/2$ we have
\begin{align*}
    2x - 4x^2 \leq \log\frac{1+x}{1-x} \leq 2x + 4x^2 \quad \text{and} \quad 
    0 \leq x\log\frac{1+x}{1-x} \leq 4x^2.
\end{align*}
\begin{proof}
From Lemma \ref{lem:ratiobound} and the inequality $\log(1 + x) \leq x$ we have the upper bound
\begin{align*}
    \log\frac{1+x}{1-x} \leq \log(1+2x+4x^2) \leq 2x+4x^2.
\end{align*}
For the lower bound, we use Lemma \ref{lem:ratiobound} and \ref{lem:logbound}:
\begin{align*}
   \log\frac{1+x}{1-x} \geq \log(1 + 2x + x^2) \geq \log(1 + 2x) \geq 2x - 4x^2
\end{align*}
when $|x| \leq 1/2$, which establishes the first result. To establish the second result, first consider $0 \leq x \leq 1/2$. In this case, the result just established gives
\begin{align*}
    x\log\frac{1+x}{1-x} \leq x(2x + 4x^2) \leq 4x^2
\end{align*}
for $x \in [0,1/2]$. When $x \in [-1/2,0)$, we use the lower bound just established to get
\begin{align*}
    x\log\frac{1+x}{1-x} \leq x(2x - 4x^2) \leq 4x^2,
\end{align*}
which establishes the upper bound for all $|x| \leq 1/2$. The product is non-negative because $x$ and $\log((1+x)/(1-x))$ always have the same sign. 
\end{proof}
\end{lem}
\begin{lem}[Bound on log squared] \label{lem:squared-log-bound}
For $0 < x < \sqrt{2}$ we have
\begin{align*}
    \left(\log{x}\right)^2 \geq \frac{(x-1)^2}{2}.
\end{align*}
\begin{proof}
First consider $0 < x < 1$. Then it suffices to establish that $0 > (x-1)/\sqrt{2} > \log{x}$, which follows from $x-1 > \log{x}$ (by convexity) for all $x > 0$ and $0 > (x-1)/\sqrt{2} > x - 1$ for $x < 1$. Now consider the case $1 \leq x < \sqrt{2}$, when it suffices to establish that $0 \leq (x-1)/\sqrt{2} \leq \log{x}$. Define the function $h(x) = \log{x} - (x-1)/\sqrt{2}$. The statement follows by noting that $h(1) = 0$ and $h'(x) = 1/x - 1/\sqrt{2} > 0$ for $x < \sqrt{2}$.

\end{proof}
\end{lem}

\begin{lem}[Log ratio squared bound]\label{lem:log-ratio-squared}
When $-1 < x < (\sqrt{2}-1)/(\sqrt{2} + 1)$ we have
\begin{align*}
    \left(\log\frac{1+x}{1-x}\right)^2 \geq \frac{x^2}{2}.
\end{align*}
\end{lem}
\begin{proof}
    When $x < (\sqrt{2} - 1)/(\sqrt{2} + 1)$ then $(1+x)/(1-x) < \sqrt{2}$ so we can apply Lemma \ref{lem:squared-log-bound} to get
    \begin{align*}
     \left(\log\frac{1+x}{1-x}\right)^2 \geq \frac{1}{2}\left(\frac{1+x}{1-x} - 1\right)^2. 
    \end{align*}
    As $0<(\sqrt{2} - 1)/(\sqrt{2} + 1) < 1/2$, we can apply Lemma \ref{lem:ratiobound} to get the desired result. 
\end{proof}

\section{Acknowledgment} We would like to thank Prof. Natesh Pillai (Harvard) for the fruitful discussions, for initiating our collaboration and hosting BSz at Harvard University for several research visits. This project has received funding from the European Research Council (ERC) under the European Union’s Horizon 2020 research and innovation programme (grant agreement No
101041064).

\bibliography{q}

\begin{thebibliography}{24}
\providecommand{\natexlab}[1]{#1}
\providecommand{\url}[1]{\texttt{#1}}
\expandafter\ifx\csname urlstyle\endcsname\relax
  \providecommand{\doi}[1]{doi: #1}\else
  \providecommand{\doi}{doi: \begingroup \urlstyle{rm}\Url}\fi

\bibitem[Acharya et~al.(2020{\natexlab{a}})Acharya, Canonne, and
  Tyagi]{pmlr-v125-acharya20b}
Jayadev Acharya, Cl{\'e}ment~L Canonne, and Himanshu Tyagi.
\newblock Distributed signal detection under communication constraints.
\newblock In Jacob Abernethy and Shivani Agarwal, editors, \emph{Proceedings of
  Thirty Third Conference on Learning Theory}, volume 125 of \emph{Proceedings
  of Machine Learning Research}, pages 41--63. PMLR, 09--12 Jul
  2020{\natexlab{a}}.
\newblock URL \url{https://proceedings.mlr.press/v125/acharya20b.html}.

\bibitem[Acharya et~al.(2020{\natexlab{b}})Acharya, Canonne, and
  Tyagi]{9211418}
Jayadev Acharya, Clément~L. Canonne, and Himanshu Tyagi.
\newblock Inference under information constraints i: Lower bounds from
  chi-square contraction.
\newblock \emph{IEEE Transactions on Information Theory}, 66\penalty0
  (12):\penalty0 7835--7855, 2020{\natexlab{b}}.
\newblock \doi{10.1109/TIT.2020.3028440}.

\bibitem[Barnes et~al.(2019)Barnes, Han, and Ozgur]{barnes_lower_2019}
Leighton~Pate Barnes, Yanjun Han, and Ayfer Ozgur.
\newblock Lower {Bounds} for {Learning} {Distributions} under {Communication}
  {Constraints} via {Fisher} {Information}.
\newblock \emph{arXiv:1902.02890 [cs, math, stat]}, May 2019.
\newblock URL \url{http://arxiv.org/abs/1902.02890}.
\newblock arXiv: 1902.02890.

\bibitem[Braverman et~al.(2016)Braverman, Garg, Ma, Nguyen, and
  Woodruff]{braverman_communication_2016}
Mark Braverman, Ankit Garg, Tengyu Ma, Huy~L. Nguyen, and David~P. Woodruff.
\newblock Communication {Lower} {Bounds} for {Statistical} {Estimation}
  {Problems} via a {Distributed} {Data} {Processing} {Inequality}.
\newblock \emph{arXiv:1506.07216 [cs, math, stat]}, May 2016.
\newblock URL \url{http://arxiv.org/abs/1506.07216}.
\newblock arXiv: 1506.07216.

\bibitem[Cai and Wei(2020)]{cai_distributed_2020}
T.~Tony Cai and Hongji Wei.
\newblock Distributed {Gaussian} {Mean} {Estimation} under {Communication}
  {Constraints}: {Optimal} {Rates} and {Communication}-{Efficient}
  {Algorithms}.
\newblock \emph{arXiv:2001.08877 [cs, math, stat]}, January 2020.
\newblock URL \url{http://arxiv.org/abs/2001.08877}.
\newblock arXiv: 2001.08877.

\bibitem[Cai and Wei(2021{\natexlab{a}})]{cai:2021:distributed}
T~Tony Cai and Hongji Wei.
\newblock Distributed nonparametric function estimation: Optimal rate of
  convergence and cost of adaptation.
\newblock \emph{arXiv preprint arXiv:2107.00179}, 2021{\natexlab{a}}.

\bibitem[Cai and Wei(2021{\natexlab{b}})]{cai:distributed:adap:sigma}
T~Tony Cai and Hongji Wei.
\newblock Distributed adaptive gaussian mean estimation with unknown variance:
  interactive protocol helps adaptation.
\newblock \emph{arXiv preprint arXiv:0000.0000}, 2021{\natexlab{b}}.

\bibitem[Cohen et~al.(1993)Cohen, Daubechies, and Vial]{CohenAlbert1993WotI}
Albert Cohen, Ingrid Daubechies, and Pierre Vial.
\newblock Wavelets on the interval and fast wavelet transforms.
\newblock \emph{Applied and computational harmonic analysis}, 1\penalty0
  (1):\penalty0 54--81, 1993.
\newblock ISSN 1063-5203.

\bibitem[Daubechies(1992)]{daubechies:1992}
I.~Daubechies.
\newblock \emph{Ten Lectures on Wavelets}.
\newblock Society for Industrial and Applied Mathematics, 1992.
\newblock \doi{10.1137/1.9781611970104}.
\newblock URL \url{https://epubs.siam.org/doi/abs/10.1137/1.9781611970104}.

\bibitem[Duchi et~al.(2014)Duchi, Jordan, Wainwright, and
  Zhang]{duchi_optimality_2014}
John~C. Duchi, Michael~I. Jordan, Martin~J. Wainwright, and Yuchen Zhang.
\newblock Optimality guarantees for distributed statistical estimation.
\newblock \emph{arXiv:1405.0782 [cs, math, stat]}, June 2014.
\newblock URL \url{http://arxiv.org/abs/1405.0782}.
\newblock arXiv: 1405.0782.

\bibitem[Fan et~al.(2019)Fan, Wang, Wang, and Zhu]{fan2019distributed}
Jianqing Fan, Dong Wang, Kaizheng Wang, and Ziwei Zhu.
\newblock Distributed estimation of principal eigenspaces.
\newblock \emph{Annals of statistics}, 47\penalty0 (6):\penalty0 3009, 2019.

\bibitem[Gin\'e and Nickl(2015)]{Gine}
Evarist Gin\'e and Richard Nickl.
\newblock \emph{Mathematical Foundations of Infinite-Dimensional Statistical
  Models}.
\newblock Cambridge University Press, New York, NY, USA, 1st edition, 2015.
\newblock ISBN 1107043166, 9781107043169.

\bibitem[Han et~al.(2020)Han, Ozgur, and Weissman]{han_geometric_2020}
Yanjun Han, Ayfer Ozgur, and Tsachy Weissman.
\newblock Geometric {Lower} {Bounds} for {Distributed} {Parameter} {Estimation}
  under {Communication} {Constraints}.
\newblock \emph{arXiv:1802.08417 [cs, math, stat]}, September 2020.
\newblock URL \url{http://arxiv.org/abs/1802.08417}.
\newblock arXiv: 1802.08417.

\bibitem[H{\"a}rdle et~al.(2012)H{\"a}rdle, Kerkyacharian, Picard, and
  Tsybakov]{hardle:2012}
W.~H{\"a}rdle, G.~Kerkyacharian, D.~Picard, and A.~Tsybakov.
\newblock \emph{Wavelets, Approximation, and Statistical Applications}.
\newblock Lecture Notes in Statistics. Springer New York, 2012.
\newblock ISBN 9781461222224.
\newblock URL \url{https://books.google.nl/books?id=8yXUBwAAQBAJ}.

\bibitem[Mukherjee and Sen(2018)]{adaptivebinaryreg}
Rajarshi Mukherjee and Subhabrata Sen.
\newblock Optimal adaptive inference in random design binary regression.
\newblock \emph{Bernoulli : official journal of the Bernoulli Society for
  Mathematical Statistics and Probability}, 24\penalty0 (1):\penalty0 699--739,
  2018.
\newblock ISSN 1350-7265.

\bibitem[Shamir(2014)]{shamir_fundamental_2014}
Ohad Shamir.
\newblock Fundamental {Limits} of {Online} and {Distributed} {Algorithms} for
  {Statistical} {Learning} and {Estimation}.
\newblock \emph{arXiv:1311.3494 [cs, stat]}, October 2014.
\newblock URL \url{http://arxiv.org/abs/1311.3494}.
\newblock arXiv: 1311.3494.

\bibitem[Szabo and van Zanten(2020{\natexlab{a}})]{szabo2020distributed}
Botond Szabo and Harry van Zanten.
\newblock Distributed function estimation: adaptation using minimal
  communication.
\newblock \emph{arXiv preprint arXiv:2003.12838}, 2020{\natexlab{a}}.

\bibitem[Szabo and van Zanten(2020{\natexlab{b}})]{szabo_adaptive_2019}
Botond Szabo and Harry van Zanten.
\newblock {Adaptive distributed methods under communication constraints}.
\newblock \emph{The Annals of Statistics}, 48\penalty0 (4):\penalty0 2347 --
  2380, 2020{\natexlab{b}}.
\newblock \doi{10.1214/19-AOS1890}.
\newblock URL \url{https://doi.org/10.1214/19-AOS1890}.

\bibitem[Szabo et~al.(2020)Szabo, Vuursteen, and van
  Zanten]{szabo_optimal_2020}
Botond Szabo, Lasse Vuursteen, and Harry van Zanten.
\newblock Optimal distributed testing in high-dimensional {Gaussian} models.
\newblock \emph{arXiv:2012.04957 [cs, math, stat]}, December 2020.
\newblock URL \url{http://arxiv.org/abs/2012.04957}.
\newblock arXiv: 2012.04957.

\bibitem[Szabo et~al.(2022)Szabo, Vuursteen, and van
  Zanten]{sz:vuursteen:zanten:2022}
Botond Szabo, Lasse Vuursteen, and Harry van Zanten.
\newblock Optimal high-dimensional and nonparametric distributed testing under
  communication constraints.
\newblock \emph{preprint}, 2022.

\bibitem[Wainwright(2019)]{Wainwright}
Martin~J Wainwright.
\newblock \emph{High-Dimensional Statistics: A Non-Asymptotic Viewpoint},
  volume~48 of \emph{Cambridge series in statistical and probabilistic
  mathematics}.
\newblock Cambridge University Press, New York, NY, 2019.
\newblock ISBN 1108498027.

\bibitem[Xu and Raginsky(2016)]{xuraginsky2016}
Aolin Xu and Maxim Raginsky.
\newblock Information-{Theoretic} {Lower} {Bounds} on {Bayes} {Risk} in
  {Decentralized} {Estimation}.
\newblock \emph{arXiv:1607.00550 [cs, math, stat]}, July 2016.
\newblock URL \url{http://arxiv.org/abs/1607.00550}.
\newblock arXiv: 1607.00550.

\bibitem[Zhang et~al.(2013)Zhang, Duchi, Jordan, and
  Wainwright]{zhang2013information}
Yuchen Zhang, John~C Duchi, Michael~I Jordan, and Martin~J Wainwright.
\newblock Information-theoretic lower bounds for distributed statistical
  estimation with communication constraints.
\newblock In \emph{NIPS}, pages 2328--2336. Citeseer, 2013.

\bibitem[Zhu and Lafferty(2018)]{pmlr-v80-zhu18a}
Yuancheng Zhu and John Lafferty.
\newblock Distributed nonparametric regression under communication constraints.
\newblock In Jennifer Dy and Andreas Krause, editors, \emph{Proceedings of the
  35th International Conference on Machine Learning}, volume~80 of
  \emph{Proceedings of Machine Learning Research}, pages 6009--6017. PMLR,
  10--15 Jul 2018.
\newblock URL \url{https://proceedings.mlr.press/v80/zhu18a.html}.

\end{thebibliography}

\section{Supplement}

\subsection{Wavelets and Besov spaces}\label{app: wavelets}
For the convenience of the reader, we include a brief overview of wavelets and Besov spaces, for more details, see \cite{hardle:2012,Gine}. Wavelets represent a particular choice of basis in which we can represent a function of interest. For practitioners, wavelets are appealing because they often give sparse representations of functions, which can be useful for tasks like image compression. 

We follow Daubechies' construction, as in \cite{daubechies:1992}, of the father wavelet $\phi$ and the mother wavelet $\psi$ with $N \in \mathbb{N}$ vanishing moments and bounded support on $[0, 2N - 1]$ and $[-N + 1, N]$, respectively. For a fixed base resolution level $j_0$ (which we will always take to be zero in our paper), we set $\phi_{j_0k}(x) = 2^{j_0}\phi(2^{j_0}x - k)$ for $ k \in \mathbb{Z}$. We then define the child wavelets $\psi_{jk}(x) = 2^{j/2}\psi(2^jx -k)$ for $k \in \mathbb{Z}$ and $j> j_0$. The set  $ \left\{\phi_{j_0  k}:  k \in \mathbb{Z}\right\} \bigcup \left\{\psi_{jk}: j > j_0, k \in \mathbb{Z}\right\}$ then forms an orthonormal multiresolution wavelet basis for $L_2(\mathbb{R})$. To simplify our expressions, we abuse notation slightly and write $\psi_{j_0k} \coloneqq \phi_{j_0k}$. This construction gives wavelets on the whole real line.

In our work we are considering square integrable functions on the unit interval $[0,1]$. The Daubechies'  wavelets can be transformed to an orthonormal wavelet basis over the domain $[0,1]$, called boundary (corrected) wavelets, see for instance \cite{CohenAlbert1993WotI} or Chapter 4.3.5 of \cite{Gine}. 
Let us denote by $\{\psi_{jk}: j \geq j_0, k =0,...,2^j-1\}$ the so constructed orthonormal multiresolution wavelet basis for $L_2([0,1])$, which allows us to compactly express any $f \in L_2([0,1])$ as
\begin{align*}
    f = \sum_{j = j_0}^\infty \sum_{k = 0}^{2^j-1} f_{jk}\psi_{jk},
\end{align*}
where $f_{jk} = \langle f, \psi_{jk}\rangle$ are the wavelet coefficients. The $L_2$ norm of $f$ is then expressed as
\begin{align*}
    \Vert f \Vert_2^2 = \sum_{j=j_0}^\infty\sum_{k=0}^{2^j - 1} f_{jk}^2.
\end{align*}

Wavelets naturally define Besov regularity spaces $B_{pq}^r$ on the unit interval $[0,1]$ equipped with the Besov norm
\begin{align*}
    \Vert f \Vert_{B_{pq}^r} =  \Big(\sum_{j = j_0}^\infty \Big\{ 2^{j\left(r + \frac{1}{2} - \frac{1}{p}\right)}\Big[\sum_{k=0}^{2^j - 1} \vert f_{jk}\vert^p\Big]^{\frac{1}{p}}\Big\}^q\Big)^{\frac{1}{q}}.
\end{align*}
Then Besov balls with radius $L$ are defined as
\begin{align*}
    B_{pq}^r(L) = \left\{f \in L_p([0,1]): \Vert f \Vert_{B_{22}^r} < L\right\}.
\end{align*}
In our analysis we consider $B_{22}^r$ Besov classes (with parameters $p=q=2$). This space is shown to be equivalent with the standard Sobolev regularity class, see for instance \cite{hardle:2012,Gine}.

\subsection{Proof of Lemma \ref{lem: help:Minimax:LB}}

We start by introducing some notations. For a finite set of functions $\mathcal{F}_0$ and semimetric $\rho$ we define the quantities
\begin{align*}
    \ntmax &= \max_{f \in \mathcal{F}_0} \left\vert \{\tilde{f} \in \mathcal{F}_0: \rho(f, \tilde{f}) \leq t\} \right\vert, \\
    \ntmin &= \min_{f \in \mathcal{F}_0} \left\vert \{\tilde{f} \in \mathcal{F}_0: \rho(f,\tilde{f}) \leq t\}  \right\vert.
\end{align*}
This allows us to state a version of Fano's lemma given in Corollary 1 of \cite{duchi_optimality_2014}.

\begin{lem}[Corollary 1 of \cite{duchi_optimality_2014}] \label{thm:fanos}
If the semimetric space $(\mathcal{F},\rho)$ contains a finite set $\mathcal{F}_0$ and $|\mathcal{F}_0| - \ntmin > \ntmax$, then for all $p, t > 0,$
\begin{align*}
    \inf_{\hat{f} \in \mathcal{E}(Y)}\sup_{f \in \mathcal{F}} \mathbb{E}_f \rho(\hat{f}, f)^p &\geq t^p\left( 1 - \frac{I(F;Y) + \log{2}}{\log(|\mathcal{F}_0|/\ntmax)}\right),
\end{align*}
where $\mathcal{E}(Y)$ denotes the set of measurable functions of $Y$, $I(F;Y)$ is the mutual information between the uniform random variable $F$ (on $\mathcal{F}_0$) and $Y$, in the Markov chain $F \to X \to Y$, and $\mathbb{E}_f$ is the expectation with respect to the distribution of $Y$ given $F = f$. 
\end{lem}

The distance between two functions in the sieve $\mathcal{F}_0(d,C_0, \epsilon)$ is
\begin{align*}
    \Vert \fbb - \fbbp \Vert_2^2 &= \epsilon^2 d^{-1-2r}\sum_{h=1}^d I(\beta_h \neq \beta_h').
\end{align*}
This means that for $\tilde{t} = \lfloor t^2 d^{1+2r}/\epsilon^{2} \rfloor$ we have for $\rho(f,g) = \Vert f-g \Vert_2$ that
\begin{align*}
    \ntmax = \ntmin = \sum_{i = 0}^{ \tilde{t}} \binom{d}{i}.
\end{align*}
Note that if $\tilde{t} < d/2$, then $\ntmax = \ntmin < 2^d$. We set $t^2 = \epsilon^2d^{-2r}/6$. This gives $\tilde{t} \leq d/6$, and  thus $\ntmax = \ntmin < |\mathcal{F}_0(d,C_0, \epsilon)|/2$ and $\ntmax < |\mathcal{F}_0(d,C_0, \epsilon)| - \ntmin$. Hence by applying Lemma \ref{thm:fanos} with $p=2$, $t^2 = \epsilon^2d^{-2r}/6$  we obtain
\begin{align*}
      \inf_{\hat{f} \in \mathcal{F}_{\text{dist}}(\boldsymbol{B})} \sup_{f \in B_{2,2}^r(L) } \mathbb{E}_{f,T}\Vert \hat{f} - f\Vert_2^2 &\gtrsim d^{-2r}\left(1 - \frac{I(F;Y) + \log{2}}{\log(|\mathcal{F}_0(d,C_0, \epsilon)|/\ntmax)}\right).
\end{align*}
Since $d^{-2r}\asymp \Nbass^{-\frac{2r}{1+2r}}$ it is sufficient to prove that the term in the bracket is bounded away from zero.

Noting, that with $\tilde{t} = \lfloor d/6 \rfloor$,
\begin{align*}
    \ntmax = \sum_{i=0}^{\tilde{t}} \binom{d}{i} < 2\binom{d}{\tilde{t}} \leq 2(ed/\tilde{t})^{\tilde{t}} \leq 2(ed/(d/7))^{d/6} = 2(7e)^{d/6}
\end{align*}
we get
\begin{align*}
    \log(|\mathcal{F}_0(d,C_0, \epsilon)|/\ntmax) \geq \log\frac{2^d}{2(7e)^{d/6}} = d\log(2(7e)^{-1/6}2^{-1/d}) > d/6
\end{align*}
for $d \geq 20$

\subsection{Algorithm for the upper bound}

Algorithm \ref{alg:block-est-alg} provides the schematic representation of the distributed method given in Theorem \ref{thm:gen_upper_bound_thm} achieving the theoretical lower bounds (up to a logarithmic factor).

\begin{algorithm}
    \caption{Estimation procedure for block estimators}
    \label{alg:block-est-alg}
    \begin{algorithmic}[1] 
        \Procedure{BlockEstimator}{$B, n, m, \hat{f}_{jh}$} 
            \State $\kappa \gets \kappa(B, n, m)$ \Comment{Calculate $\kappa$ using equation \eqref{eq:optimal-kappa}}
            \For{$\ell = 1$ to $m/\kappa$}
                \For{$k = \lfloor (\ell - 1)\kappa\rfloor + 1$ to $\lfloor \ell\kappa\rfloor$}
                    \For{$2^j + h = (\ell - 1)\lfloor B/\log_2{N}\rfloor$ to $\ell\lfloor B/\log_2{N}\rfloor$}
                        \State{\textbf{On machine $k$:}}
                        \State{$\hat{f}_{n,jh}^{(k)} \gets \hat{f}_{jh}(X^{(k)}, T^{(k)})$} \Comment{Estimate coefficient}
                        \If{$\vert\hat{f}_{n,jh}^{(k)}\vert \leq  \sqrt{N}$}
                        \State{$Y_{jh}^{(k)} \gets$ \textsc{TransApprox}$(\hat{f}_{n,jh}^{(k)})$}
                        \Else
                        \State{$Y_{jh}^{(k)} \gets 0$}
                        \EndIf
                        \State{$U_{jh} \gets U_{jh} \cup \{k\}$} \Comment{Indicies of machines used to estimate $f_{jh}$}
                    \EndFor
                \EndFor
            \EndFor
            \State \textbf{On the central machine}
            \For{$2^j + h = 1$ to $m/\kappa\lfloor B/\log_2{N}\rfloor$}
            
                \State{$\hat{f}_{jh} = \kappa^{-1}\sum_{k \in U_{jh}} Y_{jh}^{(k)}$}

            \EndFor
            \State{\textbf{return:} $\hat{f} = \sum_{2^j + h: U_{jh} \neq \emptyset}\hat{f}_{jh}\psi_{jk}$}
        \EndProcedure
    \end{algorithmic}
\end{algorithm}

\subsubsection{Binary approximation and communication constraints}
 In order to construct estimators, we need to actually specify how the communication budget will be used. To do this, we translate the local estimates $\hat{f}_{n,jh}^{(k)}$ into a binary string $Y_{jh}^{(k)}$ that can be transmitted. Given any value $x \in \mathbb{R}$, we can represent $x$ in a scientific binary representation: $|x| =  \sum_{k=-\infty}^{\log_2|x|} b_k2^k$ for $b_k \in \{0,1\}$. We can approximate $x$ using $y$ by taking the same binary expansion up to the $(D\log_2{N})$th binary digit, which gives $|y| = \sum_{k=-D\log_2{N}+1}^{\log_2{|x|}} b_k2^k$. This process is described in Algorithm \ref{alg:transapprox}, which is Algorithm 1 from \cite{szabo_adaptive_2019}. 
\begin{algorithm}
    \caption{Transmitting a finite-bit approximation of a number}
    \label{alg:transapprox}
    \begin{algorithmic}[1] 
        \Procedure{TransApprox}{$x$} 
            \State Transmit: $\text{sign}(x), b_{-D\lfloor\log_2{N}\rfloor},\ldots b_{\lfloor\log_2|x|\rfloor}$
            \State Construct: $y = (2\text{sign}(x) - 1)\sum_{k=-D\log_2{N}}^{\log_2|x|} b_k2^k$
        \EndProcedure
    \end{algorithmic}
\end{algorithm}
The length of the binary string approximation given by Algorithm \ref{alg:transapprox} for $|x|\leq \sqrt{N}$ is $ \log_2|x|+D\log_2 n\leq (D+1/2)\log_2 N$. Taking $D=1/2$, we can approximate our estimators with error bounded by $N^{-1/2}$ using $\log_2{N}$ bits. Therefore one machine can transmit $B/\log_2{N}$ binary approximations to coefficient estimates with each having error bounded by $N^{-1/2}$ and satisfy the communication constraint of $B$ bits in expectation. 

The estimation algorithm for block estimators is shown in Algorithm \ref{alg:block-est-alg}. The algorithm takes as input the communication budget $B$, the local sample size $n$, the number of machines $m$, and local estimators $\hat{f}_{n,jh}^{(k)}$ for all $j$, $h$ and $k$. By estimating wavelet coefficients, the algorithm is able to construct an estimator $\hat{f}$ for the unknown function. 

In the algorithm, we transmit zero if the estimate $\vert\hat{f}_{n,jh}^{(k)}\vert > \sqrt{N}$. This is necessary to construct estimators that satisfy the communication constraint almost surely, rather than in expectation. In practice, unless one needed to satisfy a communication constraint with probability one, there would be no need to truncate to zero.

\end{document}